\documentclass[reqno, 12pt, reqno]{amsart}

\usepackage{amsfonts, amsthm, amsmath, amssymb}
\usepackage{hyperref}
\hypersetup{colorlinks=false}

\usepackage[margin=1.5in]{geometry}

\RequirePackage{mathrsfs} \let\mathcal\mathscr

\numberwithin{equation}{section}

\newtheorem{theorem}{Theorem}[section]
\newtheorem{lemma}[theorem]{Lemma}

\theoremstyle{definition}

\renewcommand{\d}{\mathrm{d}}
\renewcommand{\phi}{\varphi}
\renewcommand{\rho}{\varrho}

\newcommand{\sumstar}{\sideset{}{^*}\sum}

\newcommand{\card}{\#}
\newcommand{\0}{\mathbf{0}}
\newcommand{\PP}{\mathbb{P}}

\newcommand{\ZZ}{\mathbb{Z}}
\newcommand{\ZZp}{\mathbb{Z}_{\mathrm{prim}}}
\newcommand{\NN}{\mathbb{N}}
\newcommand{\QQ}{\mathbb{Q}}
\newcommand{\RR}{\mathbb{R}}

\renewcommand{\leq}{\leqslant}
\renewcommand{\le}{\leqslant}
\renewcommand{\geq}{\geqslant}
\renewcommand{\ge}{\geqslant}
\renewcommand{\bar}{\overline}

\newcommand{\ma}{\mathbf}

\newcommand{\x}{\mathbf{x}}
\newcommand{\y}{\mathbf{y}}
\renewcommand{\c}{\mathbf{c}}
\renewcommand{\v}{\mathbf{v}}
\renewcommand{\u}{\mathbf{u}}
\newcommand{\z}{\mathbf{z}}
\newcommand{\w}{\mathbf{w}}
\renewcommand{\b}{\mathbf{b}}
\renewcommand{\a}{\mathbf{a}}

\newcommand{\ve}{\varepsilon}
\newcommand{\ep}{\varepsilon}

\DeclareMathOperator{\Pic}{Pic}

\DeclareMathOperator{\meas}{meas}
\DeclareMathOperator{\supp}{supp}

\DeclareMathOperator{\Hilb}{Hilb}

\renewcommand{\t}{\mathbf{t}}

\renewcommand{\hat}{\widehat}

\newcommand{\Dbad}{\Delta_{\mathrm{bad}}}
\newcommand{\beql}[1]{\begin{equation}\label{#1}}
\newcommand{\eeq}{\end{equation}}

\begin{document}

\date{\today}

\title[Density of rational points on a quadric bundle]{Density of
  rational points on a quadric bundle in $\PP^3\times \PP^3$} 
\author{T.D.\ Browning}
\author{D.R.\ Heath-Brown}

\address{IST Austria\\
Am Campus 1\\
3400 Klosterneuburg\\
Austria}
\email{tdb@ist.ac.at}

\address{Mathematical Institute\\
Radcliffe Observatory Quarter\\ Woodstock Road\\ Oxford\\ OX2 6GG\\ United Kingdom}
\email{rhb@maths.ox.ac.uk}


\begin{abstract}
An asymptotic formula is established for the number 
of rational points of bounded anticanonical height which lie on 
a certain Zariski dense subset of  the biprojective hypersurface
$$
x_1y_1^2+\dots+x_4y_4^2=0
$$
in $\PP^3\times\PP^3$.
This confirms the modified 
Manin conjecture for this variety, 
in which the removal of  a ``thin'' set of rational points is allowed.
\end{abstract}

\date{\today}

\maketitle

\thispagestyle{empty}

\setcounter{tocdepth}{1}
\tableofcontents

\section{Introduction}

The main goal of this paper is to study the density of rational points on the 
biprojective 
hypersurface $X\subset \PP^3\times\PP^3$ cut out by the equation 
$F(\x;\y)=0$, where
$$
F(\x;\y)=x_1y_1^2+\dots+x_4y_4^2.
$$
Let $H: X(\QQ)\to \RR_{\geq 0}$ be 
an anticanonical height function 
and let 
\beql{NXTB}
N(\Omega,B)=
\#\left\{ (x,y)\in \Omega:   H(x,y)\leq B \right\},
\eeq
for any subset $\Omega\subset X(\QQ)$.
The variety $X$ defines a smooth hypersurface of bidegree $(1,2)$ and
has  Picard group  $\Pic(X)\cong \ZZ^2$. If a point $(x,y)\in X(\QQ)$
is represented by a vector $(\x,\y)\in\ZZp^4\times\ZZp^4$, then
we shall take 
$H(x,y)=|\x|^3|\y|^2$, where 
$|\cdot|: \RR^4\to \RR_{\geq 0}$ is the sup-norm. 
In view of the Manin Conjecture~\cite{FMT89}, one might  expect that
there is a Zariski open subset $U\subset X$ such that 
$$N(U(\QQ),B)\sim cB\log B,$$ 
as $B\to\infty$, where $c$ is the constant
predicted by Peyre \cite{Pey95}.  
We certainly require $U$ to exclude all subvarieties of the form
$x_i=x_j=x_k=y_l=0$ for  $\{i,j,k,l\}=\{1,2,3,4\}$,
since the rational points on $X$  which satisfy these constraints  
are easily seen to contribute $\gg B^{3/2}$ to $N(X(\QQ),B)$.
Similarly, we get a contribution of the same order of magnitude from
rational points for which
$\x=(0,0,1,-1)$ and $\y=(a,b,c,c)$, for example. 

More interestingly, any fibre $X_x=\pi_1^{-1}(x)$ over a point $x\in
\PP^3(\QQ)$ such that $X_x \cong \PP^1\times \PP^1$ 
will  contribute $\sim c_x B\log B$, as $B\to \infty$, for an
appropriate constant $c_x>0$ depending on $x$.   
 It is expected that the total contribution from these rational points
 will be $\sim a B\log B$, where 
$$
a=\sum_{\substack{x\in \PP^3(\QQ)\\ X_x \cong \PP^1\times \PP^1}} c_x
$$
is a  convergent series. 
Note that if $x=
[x_1,\dots,x_4]
\in \PP^3(\QQ)$ then 
the isomorphism
$X_x \cong \PP^1\times \PP^1$ holds if and only if  $X_x(\QQ)\neq
\emptyset$ and   $x_1\dots x_4$ 
is  a  square in $\QQ^*$.
In view of this we are led to study 
\eqref{NXTB} when $\Omega$ is
obtained from $X(\QQ)$ by deleting  the set 
\begin{equation}\label{eq:TT}
T=\left\{(x,y)\in X(\QQ): x_1\dots x_4=\square\right\}.
\end{equation}
 Our main result is then the following.

\begin{theorem}\label{t:main}
Let $\Omega=X(\QQ)\setminus T$. Then 
  \[N(\Omega,B)\sim cB\log B\]
  as $B\to\infty$, where
  \begin{equation}\label{PC}
  c=\frac{\tau_\infty}{4 \zeta(3)\zeta(4)} 
\end{equation}
  is the Peyre constant for the variety $X$, with 
\begin{equation}\label{eq:tau_infty}
  \tau_\infty=\int_{-\infty}^\infty \int_{[-1,1]^8} 
  e(-\theta F(\x;\y))\d\x\d\y\d\theta.  
\end{equation}
 \end{theorem}
 
 We shall see in Lemma \ref{tinf} that the integral
\[\int_{-\infty}^\infty \left|\int_{[-1,1]^8} 
  e(-\theta F(\x;\y))\d\x\d\y\right|\d\theta.  \]
is convergent.

     Theorem \ref{t:main} answers a question that was originally   
 raised by Colliot-Th\'el\`ene, and mentioned in work of  
Batyrev and Tschinkel    \cite[Ex.~3.5.3]{BT98}  
over 20 years ago. 
  The set $T$ in \eqref{eq:TT} is 
an example of a ``thin'' set of rational points, as introduced to the
subject by  
Serre  \cite[\S 3.1]{Ser08}. Theorem~\ref{t:main} therefore  confirms
the  refined Manin  
conjecture for $X$, in which one is allowed to remove a finite number
of thin sets.  
Lehmann,  Sengupta and Tanimoto \cite{geometry} have 
developed a geometric framework for identifying the relevant thin sets 
for any Fano variety.  Follow-up work of Lehmann and Tanimoto \cite[Thm.~12.6]{LT} confirms that 
our  set $T$ is compatible with their predictions for the quadric bundle $X$.

Our result adds to the small store of examples in which thin sets have
been shown to exert a demonstrable influence on the distribution of
rational points on  
Fano varieties. One of the first examples in this  vein was 
discovered by
Batyrev and Tschinkel \cite{BT96}, who showed that the split cubic
surfaces in the biprojective hypersurface 
$\{x_1y_1^3+\dots+x_4y_4^3=0\}\subset  \PP^3\times\PP^3$  
contribute significantly more
than the Manin conjecture would predict for the 
number of rational points of bounded anticanonical height.
More recently,  Le Rudulier \cite{LeR13} has investigated 
Manin's conjecture for the Hilbert 
schemes $\Hilb^2(\PP^1\times\PP^1)$ and  
$\Hilb^2(\PP^2)$, with the outcome that 
a thin set of rational points needs to be removed in order for the
associated counting functions to behave as they should.

The basic line of attack in the proof of
 Theorem \ref{t:main}  involves counting points on $X$ as a union of
planes when $\y$ is small, and as a union of quadric surfaces when
$\x$ is small. In the first case $\x$ lies in a lattice determined by
$\y$, and we will use  counting arguments
that come from the
geometry of numbers.
In
the second case we can count vectors $\y$ using the circle method,
taking care to control the dependence of the error terms on $\x$. 
It
turns out that we can handle the case $|\y|\le B^{1/4}$, giving an
asymptotic formula,  using
lattices.  Moreover we can deal with the range
$B^{\delta}\le|\x|\le B^{1/6-\delta}$ via the circle
method, for any fixed $\delta>0$.  
In terms of the inequality $|\x|^3|\y|^2\leq B$, 
this leaves two small ranges
uncovered, and here it will suffice to use an upper bound of the
correct order of magnitude.  Indeed such an upper bound is also 
indispensable as an auxiliary tool in the treatment of the lattice point
counting problem. The range $B^{1/6-\delta}\le|\x|\le B^{1/6}$
contributes $O(\delta B\log B)$ to $N(\Omega,B)$, and this is
$o(B\log B)$ when we allow $\delta$ to tend to 0.  
However, we do not 
obtain an explicit error term, though it would be possible in
principle to do so, by examining more closely the dependence on
$\delta$ in our other estimates. 
One might speculate that one could prove
a version of Lemma \ref{together}
 in which the second error term had an explicit dependence on 
$\eta$, perhaps in the form $O(\eta^{-K}B^{1-\eta^2/20}\log B)$ for some 
numerical constant $K$. If that were indeed possible, then one could prove Theorem \ref{t:main} with an error term saving a positive power of $\log B$.

This  paper 
is naturally arranged in three main parts.  We begin by
discussing upper bounds in \S  \ref{s:upper}.  We go on to use these
in proving our asymptotic formula for the range $|\y|\le B^{1/4}$,
using lattice 
point counting in \S \ref{s:lattice}.  Thirdly, we develop our circle
method argument in \S \ref{CPOQ}, to 
deal with values $|\x|\le B^{1/6-\delta}$.  Once all this is in place, it remains
in \S \ref{s:combine}
and \S \ref{s:reckon} to 
combine the various results 
and consider the overall leading constant
that arises.

\subsection*{Acknowledgements}
During the preparation of this  paper the authors were
supported by the NSF under Grant No.\ DMS-1440140,  while  in residence at 
the {\em  Mathematical Sciences Research Institute} in Berkeley, California,
during the Spring 2017 semester.
The authors are very grateful to the referee for numerous helpful comments.  
While working on this paper the first author was supported   
by EPSRC grant \texttt{EP/P026710/1}.

\section{Upper bounds}\label{s:upper}

We begin by introducing some notation.  
Throughout our work we shall write $(a_1,\dots,a_k)$ for the greatest common divisor of the integers
$a_1,\dots,a_k$. We trust that any confusion with vector notation will be obviated by context.
Let $\Delta(\x)=x_1x_2x_3x_4$.
For $Y\ge 1$ and any fixed
$\x\in \ZZ^4$ we let 
$$
M_1(\x;Y)=\#\left\{ \y\in \ZZp^4: |\y|\leq Y, ~ F(\x;\y)=0  \right\}.
$$
We then set
\[M_2(X,Y)=\sum_{\substack{\x\in \ZZ^4\\ \Delta(\x)\not=\square,\,|\x|\leq X}}
M_1(\x;Y),\]
so that $M_2(X,Y)$ counts solutions $(\x,\y)\in\ZZ^4\times\ZZp^4$ of
$F(\x;\y)=0$ in the region $|\x|\le X$, $|\y|\le Y$, such that
$\Delta(\x)\not=\square$. 
Similarly, we write
\[M_3(X,Y)=\sum_{\substack{\x\in (\ZZ_{\not=0})^4\\ |\x|\leq X}} M_1(\x;Y).\]
The primary result in this section is the following collection of
upper bounds. 

\begin{lemma}\label{lem:stage1}
We have
\[M_2(X,Y)\ll X^3Y^2 +X^5Y^{2/3} \]
and
\[M_3(X,Y)\ll_{\ve} X^3Y^2+X^5Y^{2/3}+X^{2+\ve}Y^{2+\ve}, \]
for any $\ve>0$. Moreover, if $1\le X_1\le X$ then
\[\sum_{\substack{\x\in (\ZZ_{\not=0})^4\\ |\x|\leq X,\, |x_1|\le X_1}}
M_1(\x;Y)
\ll_{\ve} (XY)^\ve X_1^{3/4}X^{-3/4}\{X^3Y^2+X^4Y\}, \] 
for any $\ve>0$.
\end{lemma}

The principal tool that we will use is the authors'
result \cite[Thm.~1.1]{conics}. To state this we
introduce the arithmetic function
\begin{equation}\label{eq:phid}
\varpi(m)=\prod_{p\mid m}(1+p^{-1}),
\end{equation}
along with the notation
\begin{equation}\label{eq:Delta}
\Dbad(\x) = \prod_{\substack{p^e\| x_1\dots x_4 \\ e\geq 2}} p^e
\end{equation}
for $\x\in(\ZZ_{\not=0})^4$. We then have the following.

\begin{lemma}\label{t:upper}
Let $\x\in (\ZZ_{\not=0})^4$ and $\Dbad(\x)\le Y^{1/20}$. Then 
\[M_1(\x;Y)\ll \varpi(\Delta(\x))\Dbad(\x)^{1/3}
\left(\frac{|\x|^4}{|\Delta(\x)|}\right)^{5/8} 
L(\sigma_Y,\chi)\left(Y^{4/3}+\frac{Y^2}{|\Delta(\x)|^{1/4}}\right),\]
where $\sigma_Y=1+\frac{1}{\log Y}$ 
and $\chi(n)$ is the Dirichlet character defined by taking $\chi(2)=0$ and
\[\chi(p)=\left(\frac{\Delta(\x)}{p}\right)\]
for odd primes $p$. The implied constant in this estimate is absolute.
\end{lemma}

In fact \cite[Thm.~1.1]{conics} records an upper bound with $\prod_{p\leq Y} (1+\chi(p)/p)$ in 
place of $L(\sigma_Y,\chi)$. However, the argument given in \cite[page 3]{conics} ensures that the 
product  is at most an absolute constant multiple of $L(\sigma_Y,\chi)$.

In order to apply 
Lemma \ref{t:upper} 
to Lemma \ref{lem:stage1} we will need to
handle vectors with $\Dbad(\x)> Y^{1/20}$ via a separate auxiliary
bound.  Indeed various auxilliary bounds will be used elsewhere in the
proof of Theorem \ref{t:main}, and it is therefore natural to begin this
section by dealing with these.

\subsection{Auxiliary upper bounds}

We begin by recording a uniform upper bound for the counting function
for rational points on quadric surfaces.
The following result is due to  Heath-Brown \cite[Thm.~2]{annal}.

\begin{lemma}\label{ann}
For any irreducible quadratic form $Q(\y)\in\ZZ[y_1,\ldots,y_4]$ and
any $\ve>0$ we have
\[\#\left\{ \y\in\ZZ^4: |\y|\leq Y, ~ Q(\y)=0  \right\}\ll_{\ve}Y^{2+\ve}.\]
\end{lemma}

We next examine two counting problems involving fewer
than 4 terms.

\begin{lemma}
\label{lem:2}
Let $X,Y \geq 1$. Then 
\begin{align*}
\#\{(x_1,y_1,x_2,y_2)\in (\ZZ_{\neq 0})^4: |x_i|\leq X, ~&|y_i|\leq
Y, ~x_1y_1^2 =x_2y_2^2\} =O(XY).
\end{align*}
\end{lemma}

\begin{proof}
We extract common factors $h=(x_1,x_2)$ and $k=(y_1,y_2)$. 
For each $h$ and $k$ we are left with counting non-zero integers
$x_i',y_i'$ with  
$|x_i'|\leq X/h$ and $|y_i'|\leq Y/k$, such that
$x_1'=\pm{y_2'}^2$ and $x_2'=\pm{y_1'}^2$.  
The number of such integers is 
\[\ll \left(\min\left\{\sqrt{X/h}, Y/k\right\}\right)^2\le
(\sqrt{X/h})^{4/3}(Y/k)^{2/3}.\]
Summing over $h\leq X$ and $k\leq Y$ gives $O(XY)$ as required.
\end{proof} 

\begin{lemma}\label{l3}
Let $X,Y,U,V\ge 1$ with $XY^2=UV^2$.  Define
$T(X,Y,U,V)$ to be the number of solutions 
$\x,\y\in(\ZZ_{\not=0})^3$ of the equation
\[x_1y_1^2+x_2y_2^2+x_3y_3^2=0,\]
with 
\[|x_1|,|x_2|\le X,\;|x_3|\le U,\;|y_1|,|y_2|\le Y,\; |y_3|\le V. \]
Then $T(X,Y,U,V)\ll_{\ep}XUVY^{\ep}$ for any  $\ep>0$.
\end{lemma}

\begin{proof}
For the proof we first estimate the quantity $T^*(X,Y,U,V)$ which
counts pairs of vectors $\x,\y$ as for $T(X,Y,U,V)$, but with the added
restrictions that $\x$ and $\y$ should be primitive.
According to Heath-Brown \cite[Lemma 3]{HB} there are
\begin{align*}
&\ll 1+\frac{X^2U}{\max(Xy_1^2,Xy_2^2,Uy_3^2)} \\
  &\ll 1+X^2U(Xy_1^2)^{-1/3}(Xy_2^2)^{-1/3}(Uy_3^2)^{-1/3} 
\end{align*}
primitive solutions $\x$ for each primitive $\y$.  Summing over $\y$
and using the relation $XY^2=UV^2$ then yields
\beql{b1}
T^*(X,Y,U,V)\ll Y^2V+XUV.
\eeq
In particular $T^*(X,Y,U,V)\ll XUV$ if $Y^2\le X$.

Alternatively we may
use Corollary 2 of Browning and Heath-Brown \cite{BHB}. This shows that for any
given $\x\in(\ZZ_{\not=0})^3$ there are
\[\ll_{\ep} \left\{1+
\left(\frac{Y^2V(x_1x_2,x_1x_3,x_2x_3)^{3/2}}{|x_1x_2x_3|}\right)^{1/3}\right\}
(XY)^{\ep}\]
corresponding primitive $\y$. Let $n=|x_1x_2x_3|$ and
$d=(x_1x_2,x_1x_3,x_2x_3)$, so that $d^3\mid n^2$.
Summing over the $\x$ then gives an estimate
\begin{align*}
  T^*(X,Y,U,V)&\ll_{\ep}X^2U
  (XY)^{\ep}
+(Y^2V)^{1/3}(XY)^{2\ep}\sum_{d\le X^2U}
 d^{1/2} \sum_{\substack {n\le X^2U\\ d^3\mid n^2}}n^{-1/3}\\
&\ll_{\ep}X^2U
  (XY)^{\ep} 
+(Y^2V)^{1/3}(XY)^{2\ep}\sum_{d\le X^2U}
   d^{1/2}(X^2U)^{2/3}d^{-3/2}\\
&\ll_{\ep} \{X^2U+(Y^2V)^{1/3}(X^2U)^{2/3}\}(XY)^{3\ep}\\
&=\{X^2U+XUV\}(XY)^{3\ep}.
\end{align*}
In particular
$T^*(X,Y,U,V)\ll \{X^2U+XUV\}Y^{9\ep}$
if $Y^2\ge X$, and by our remark above the same is true in the
alternative case $Y^2\le X$ as well.

Comparing this bound with (\ref{b1}) shows that
\[T^*(X,Y,U,V)\ll_{\ep}\{\min(X^2U,Y^2V)+XUV\}Y^{9\ep}\]
whether $Y^2\ge X$ or not.
However
\[\min(X^2U,Y^2V)\le (X^2U)^{2/3}(Y^2V)^{1/3}=XUV,\]
whence
\[T^*(X,Y,U,V)\ll_{\ep}XUVY^{9\ep}.\]
We then find that
\begin{align*}
  T(X,Y,U,V)&\le \sum_{d\le X}\sum_{e\le Y}T^*(X/d,Y/e,U/d,V/e)\\
  &\ll_{\ve}\sum_{d\le X}\sum_{e\le Y}XUVY^{9\ep}d^{-2}e^{-1-9\ve}\\
  &\ll_{\ve} XUVY^{9\ep}.
\end{align*}
The lemma now follows on redefining $\ve$.
\end{proof}

Lemmas \ref{lem:2} and \ref{l3} have the following corollary.
\begin{lemma}\label{lem:xz}
  Let $X,Y\ge 1$.  Then the number of solutions $\x\in\ZZ^4$,
  $\y\in\ZZp^4$ of the equation $F(\x;\y)=0$, lying in the region
  $|\x|\le X$, $|\y|\le Y$, and satisfying the constraint
  that $\prod_{i=1}^4 x_iy_i=0$, is $O_{\ve}(X^3Y^{1+\ve}+Y^4)$, for any
  $\ve>0$.
  \end{lemma}
\begin{proof}
Suppose firstly that exactly one of the products $x_iy_i$ vanishes,
  $x_4y_4=0$ say. Then there are $O(X+Y)$ choices for $x_4$ and
$y_4$, and $O_{\ve}(X^2Y^{1+\ve})$ choices for the remaining
  variables, by Lemma \ref{l3}. Thus the total contribution is
  $O_{\ve}(X^3Y^{1+\ve}+X^2Y^{2+\ve})$. This is satisfactory
  for the lemma after redefining $\ve$, since
  $X^2Y^{2+\ve}\le\max(X^3Y^{1+2\ve},Y^4)$.

  Suppose next that exactly two of the terms $x_iy_i$ vanish,
  $x_3y_3=x_4y_4=0$, say. There are then $O(X^2+Y^2)$ choices for
  $x_3,y_3,x_4$ and $y_4$. Moreover there are $O(XY)$ choices for the
  remaining variables, by Lemma \ref{lem:2}. We therefore have a total
  $O(X^3Y+XY^3)$, which is satisfactory since $XY^3\le\max(X^3Y,Y^4)$.

  It is not possible for exactly three terms $x_iy_i$ to vanish, when
  $F(\x;\y)=0$, so the only remaining case is that in which $x_iy_i=0$
  for every index $i$. Since $\y$ is primitive it cannot vanish, and
  hence there are $O(X^3Y+Y^4)$ possibilities in this situation.  This
  completes the proof of the lemma. 
\end{proof}

The final case to consider is that in which the product
$x_1\dots x_4$ is non-zero but has a large square divisor. Let
\[\Dbad(\x)=\prod_{\substack{p^e\| x_1\dots x_4\\ e\geq 2}} p^e,\]
as in \eqref{eq:Delta}. The  following result shows that vectors $\x$
with a large value of $\Dbad(\x)$ make a small 
contribution to $M_3(X,Y)$.

\begin{lemma}\label{lem:bad_coeffs}
Let $D\geq 1$ and let $\ve>0.$ Then 
\[\sum_{\substack{\x\in (\ZZ_{\neq 0})^4 \\ |\x|\leq X\\ \Dbad(\x)\geq D}}
M_1(\x;Y)\ll_\ve (XY)^\ve\left\{ \frac{X^{3}Y^{2}}{D^{1/24}}+X^4Y\right\}.\]
\end{lemma}

In general, if we write $s(n)$ for the largest square-full divisor of
$n$, then $s(uv)\mid s(u)s(v)(u,v)^2$, as one sees by considering the case
in which $u$ and $v$ are powers of the same prime. Thus
\[s(x_1x_2x_3x_4)\mid
s(x_1)s(x_2)s(x_3)s(x_4)(x_1x_2,x_3x_4)^2(x_1,x_2)^2(x_3,x_4)^2,\]
and since
\[(x_1x_2,x_3x_4)\,\mid\, (x_1,x_3)(x_1,x_4)(x_2,x_3)(x_2,x_4)\]
we see that if $\Dbad(\x)\ge D$ then either $(x_i,x_j)\ge D^{1/24}$
for some pair of indices $i\not= j$, or $s(x_i)\ge D^{1/8}$ for some
$i$. In the latter case $d^2\mid  x_i$ for some $d\ge D^{1/24}$. Hence
\[\sum_{\substack{\x\in (\ZZ_{\neq 0})^4 \\ |\x|\leq X\\ \Dbad(\x)\geq D}} M_1(\x;Y)\leq
\sum_{d\geq D^{1/24}}\sum_{\substack{\x}}M_1(\x;Y),\]
where the 
$\x$-summation is over $\x\in (\ZZ_{\neq 0})^4$ such that $|\x|\leq X$, with 
$d^2\mid x_i$ or $d\mid (x_i,x_j)$, for some choice of distinct
indices $i,j\in \{1,2,3,4\}$.

For any $k\in \NN$ we write 
\begin{equation}\label{eq:Salpha}
  S_k(\alpha)=S_k(\alpha;X)
  =\sum_{\substack{0<|x|\le X\\ k\mid x}}\sum_{|y|\le Y}e(\alpha xy^2).
\end{equation} 
Then 
\[\sum_{\substack{\x}}M_1(\x;Y)\ll  I_1(d)+ I_2(d),\]
where
\[ I_1(d)=\int_0^1 S_d(\alpha)^2S_1(\alpha)^2\d\alpha
\quad \text{ and }\quad
 I_2(d)=\int_0^1 S_{d^2}(\alpha)S_1(\alpha)^3\d\alpha.\]
We note that $ I_j(d)=0$ unless $d^j\leq X$.  In estimating these
integrals we are led to an auxiliary counting problem, treated in the
following result. 

\begin{lemma}\label{lem:Lk}
Let $k\in \NN$ and let $X,Y\geq 1$.
  Let $L_k(X,Y)$ denote  the number of 
$(x,y,y',x_1,y_1,x_2,y_2)\in \ZZ^7$ such that $k\mid x$, with 
\[0<|x|,|x_1|,|x_2|\leq X, \quad |y|,|y'|,|y_1|,|y_2|\leq Y\]
and
\[x_1y_1^2-x_2y_2^2=x(y-y')(y+y').\]
Then for any $\ve>0$ we have 
\[L_k(X,Y)\ll_\ve (XY)^\ve\left\{ \frac{X^{2}Y^{2}+X^3Y}{k}+kX^2\right\}.\]
\end{lemma}

We will establish this in a moment, but first we show how it may be used
to complete the proof of Lemma \ref{lem:bad_coeffs}.  In general we have
\begin{equation}\label{eq:cs}
\begin{split}
|S_d(\alpha)|^2\le ~&
\#\{x\in \ZZ_{\not=0}\cap [-X,X]: d\mid x\} \\
&\hspace{3cm}\times
\sum_{\substack{0<|x|\le X\\ d\mid x}}\sum_{|y|,|y'|\le Y}e(\alpha x(y^2-y'^2)),
\end{split}
\end{equation}
by Cauchy--Schwarz. We therefore deduce from Lemma \ref{lem:Lk} that 
\begin{align*}
\sum_{d\geq D^{1/24}}
 I_1(d)&\leq \sum_{D^{1/24}\leq d\leq X}\frac{2X}{d} L_d(X,Y)\\
&\ll_\ve \sum_{D^{1/24}\leq d\leq X}
(XY)^\ve\left\{ \frac{X^{3}Y^{2}+X^4Y}{d^2}+X^3\right\}\\
&\ll_\ve
(XY)^{\ve}\left\{ \frac{X^{3}Y^{2}+X^4Y}{D^{1/24}}+X^4\right\}.
\end{align*}
This is satisfactory for Lemma \ref{lem:bad_coeffs}.

To handle $ I_2(d)$ we apply the Cauchy--Schwartz inequality, yielding
\beql{tf}
| I_2(d)|^2\leq  
\left( \int_0^1 |S_{d^2}(\alpha)|^2|S_1(\alpha)|^2\d\alpha
\right)\left( \int_0^1 |S_1(\alpha)|^4\d\alpha \right).
\eeq
We proceed as before,  
using \eqref{eq:cs} and Lemma \ref{lem:Lk} to deduce that 
\[\int_0^1 |S_{d^2}(\alpha)|^2|S_1(\alpha)|^2\d\alpha \leq
\frac{2X}{d^2}L_{d^2}(X,Y)\ll_\ve 
(XY)^\ve\left\{ \frac{X^{3}Y^{2}+X^4Y}{d^4}+X^3\right\}.\]
Taking $d=1$ we see that the second factor in (\ref{tf}) is
\begin{equation}\label{eq:orkney}
\int_0^1 |S_1(\alpha)|^4\d\alpha 
\ll_{\ve}(XY)^{\ve}(X^3Y^2+X^4Y). 
\end{equation} 
Thus 
\[ I_2(d)\ll_{\ve}
(XY)^\ve\left\{\frac{X^3Y^2+X^4Y}{d^2}+X^3Y+X^{7/2}Y^{1/2}\right\},\]
whence
\begin{align*}
\sum_{d\geq D^{1/24}} I_2(d) 
&\ll_\ve \sum_{D^{1/24}\leq d\leq \sqrt{X}} 
(XY)^{\ve}\left\{ \frac{X^3Y^2+X^4Y}{d^2}+X^3Y+X^{7/2}Y^{1/2}\right\}\\
&\ll_\ve 
(XY)^{\ve}\left\{ \frac{X^3Y^2+X^4Y}{D^{1/24}}
+X^{7/2}Y+X^{4}Y^{1/2}
\right\}.
\end{align*}
This too is satisfactory for Lemma \ref{lem:bad_coeffs}.

\begin{proof}[Proof of Lemma \ref{lem:Lk}]
Clearly 
$L_k(X,Y)=0$ unless $k\leq X$, which we now assume.
There are $O(k^{-1}XY)$ triples with $x(y-y')(y+y')=0$, and for each there
are $O(XY)$ corresponding quadruples $x_1,y_1,x_2,y_2$ with
$y_1y_2\not=0$, by Lemma \ref{lem:2}, and there are $O(X^2)$
quadruples with $y_1y_2=0$.  This case therefore
contributes a total $O(k^{-1}(X^2Y^2+X^3Y))$ to $L_k(X,Y)$.  

When $x(y-y')(y+y')\not=0$, we get
\[\ll \tau_3(|x_1y_1^2-x_2y_2^2|)\ll_{\ve} (XY)^{\ve}\]
solutions $x,y,y'$.
It therefore remains to count the number of 
$x_1,y_1,x_2,y_2$ for which $x_1y_1^2\equiv x_2y_2^2\bmod{k}$.
Breaking into residue classes modulo $k$ we deduce that  
\begin{equation}\label{eq:Lk-waypoint}
L_k(X,Y)\ll_\ve 
\frac{X^2Y^2+X^3Y}{k} +
(XY)^\ve \frac{X^2}{k^2} \left(1+\frac{Y}{k}\right)^2 \rho(k),
\end{equation}
where 
$$
\rho(k)=\card\{(x_1,x_2,y_1,y_2)\in (\ZZ/k\ZZ)^4:
x_1y_1^2\equiv x_2y_2^2\bmod{k}\}.
$$
Since $\rho(k)$ is a multiplicative arithmetic  function it suffices to estimate
$\rho(p^e)$. When $p\nmid y_1$ the values of $x_2,y_1,y_2$ determine
$x_1$, so that there are at most $p^{2e}\phi(p^e)$ such solutions. The
same argument applies when $p\nmid y_2$ so that there are at most
$2p^{2e}\phi(p^e)$ solutions with $p\nmid(y_1,y_2)$. If $e\ge 2$ there
are $p^6\rho(p^{e-2})$ solutions with $p\mid (y_1,y_2)$,
while if $e=1$ there are $p^2$ solutions. Hence
$\rho(p)\le 2p^2(p-1)+p^2\leq 2p^3 $ and
\[\rho(p^e)\le 2p^{3e}(1-p^{-1})+p^6\rho(p^{e-2})\] 
for $e\ge 2$. One can now show that $\rho(p^e)\le (e+1)p^{3e}$ for all 
$e\ge 1$, by induction. We then have $\rho(k)\ll_\ve k^{3+\ve}$
for any $\ve>0$. We therefore complete the proof of the lemma
by inserting this into \eqref{eq:Lk-waypoint} and redefining $\ve$.
\end{proof}

\subsection{Proof of Lemma \ref{lem:stage1}}

Using Lemma~\ref{t:upper} we will establish
the following result. 

\begin{lemma}\label{lem:roger}
We have  
\[\sum_{\substack{\x\in (\ZZ_{\neq 0})^4\\ |\x|\leq X, ~ 
\Delta(\x)\neq \square\\  
\Dbad(\x)\leq Y^{1/20}}} 
M_1(\x;Y) \ll  
X^{3}Y^2+X^{4}Y^{4/3}\ll X^{3}Y^2+X^{5}Y^{2/3}.\] 
\end{lemma}

\begin{proof}
The second inequality follows since
$X^4Y^{4/3}\leq \max(X^3Y^2, X^5Y^{2/3})$. To prove the first inequality,
we begin by considering dyadic ranges
\begin{equation}\label{reg}
X_i/2<|x_i|\le X_i, \quad \text{for $1\le i\le 4$}
\end{equation}
and denote the corresponding contribution 
$M(X_1,\dots,X_4;Y)$. Then, on writing $\hat{X}=X_1X_2X_3X_4$ and
\[c(X_1,\dots,X_4;Y)=\left(\frac{\max X_i^4}{\hat{X}}\right)^{5/8}
\left(Y^{4/3}+\frac{Y^2}{\hat{X}^{1/4}}\right),\]
it follows from Lemma \ref{t:upper}
that 
\[M(X_1,\dots,X_4;Y)\ll c(X_1,\dots,X_4;Y)
\sum_{\x}\varpi(\Delta(\x))
\Dbad(\x)^{1/3}L(\sigma_Y,\chi),\]
where 
$\varpi$ is given by \eqref{eq:phid} and the sum is for
$\x\in\ZZ^4$ in the region (\ref{reg})
such that $\Delta(\x)\neq \square$. Note that we are free to include
vectors with $\Dbad(\x)>Y^{1/20}$ on the right, since $L(\sigma_Y,\chi)>0$.

By checking the inequality at prime powers, one easily sees that 
\[\varpi(n)
\prod_{\substack{p^e\| n\\ e\geq 2}}p^{e/3}
\le \sum_{s,t|n} \frac{s^{1/3}}{t},\]
where $s$ and $t$ run over square-full and square-free integers
respectively. It follows that
\begin{equation}\label{quant}
 \sum_{\x}\varpi(\Delta(\x))
\Dbad(\x)^{1/3}L(\sigma_Y,\chi)\leq
\sum_{s,t}\frac{s^{1/3}}{t}
\sum_{\substack{d_1,d_2,d_3,d_4\in \NN\\ d_1d_2d_3d_4=[s,t]}}\;
  \sum_{\x\in S}L(\sigma_Y,\chi),
  \end{equation}
  where $s$ and $t$ run over square-full and square-free integers
  respectively, and
  $S=S(X_1,\dots,X_4;d_1,\dots,d_4)$ is the set of
  $\x\in\ZZ^4$ in the region (\ref{reg}) such that
  $\Delta(\x)\not=\square$ and $d_i\mid x_i$ for $1\leq i\leq 4$.

Let $\hat{d}=d_1d_2d_3d_4$.
We claim that 
\begin{equation}\label{lem:L}
\sum_{\x\in S}L(\sigma,\chi)\ll \hat{X}\hat{d}^{-7/8},
\end{equation}
uniformly for $\sigma\ge 1$.
We will prove this later, but first we observe that we can now
estimate (\ref{quant}) as
\[\ll \hat{X}\sum_{s,t}\frac{s^{1/3}}{t}\tau_4([s,t])[s,t]^{-7/8}.\]
The infinite sum has an Euler product, with factors of the shape
\[1+4p^{-1-7/8}+\sum_{e\ge 2}\sum_{f=0,1}p^{e/3-f}\tau_4(p^e)p^{-7e/8}
=1+O(p^{-13/12}).\]
The resulting product therefore converges, so that (\ref{quant}) is
$O(\hat{X})$. We then see that
\begin{align*}
M(X_1,\dots,X_4;Y)&\ll  c(X_1,\dots,X_4;Y)\hat{X}\\
&=
\left(\max X_i\right)^{5/2}\hat{X}^{3/8}
\left(Y^{4/3}+\frac{Y^2}{\hat{X}^{1/4}}\right).
\end{align*}
On summing over dyadic values for the $X_i$ we obtain Lemma \ref{lem:roger}.

It remains to prove  \eqref{lem:L}. Our key tool for the proof is a
form of Burgess' estimate \cite{burgess2}.  If 
$\theta>3/16$ we have
\[\sum_{n\le N}\psi(n)\ll_{\theta}N^{1/2}q^{\theta},\]
where $\psi$ is any non-principal character of modulus $q$. 
We  obtain the same bound for the corresponding sum over all 
integers $n$ such that $|n|\leq N$.
For our purposes it will be enough to take $\theta=1/5$ in these estimates.

The character $\chi$ is non-principal, with modulus $O(\hat{X})$. By the
Burgess bound coupled with  partial summation, we see that
\[\sum_{n>N}\frac{\chi(n)}{n^{\sigma}}\ll N^{1/2-\sigma}\hat{X}^{1/5}
\ll N^{-1/2}\hat{X}^{1/5}.\]
It follows that terms with $n> \hat{X}^{2/5}$ contribute $O(1)$ to
$L(\sigma,\chi)$, which is satisfactory 
since $\hat d\leq \hat X$. 

We proceed to consider the terms with $n\le\hat{X}^{2/5}$. Suppose that
$X_i/d_i$ is largest for $i=1$, say.  If we write $x_1=d_1q$ then
\[\sum_{\x\in S}\sum_{n\le\hat{X}^{2/5}}\frac{\chi(n)}{n^{\sigma}}=
\sum_{x_2,x_3,x_4}
\sum_{\substack{n\le \hat{X}^{2/5}\\ \text{$n$ odd}}} 
\frac{1}{n^{\sigma}}
\left(\frac{d_1x_2x_3x_4}{n}\right)\sum_{q}
\left(\frac{q}{n}\right),\]
where the sum over $q$ is for integers with 
$X_1/2d_1<|q|\le X_1/d_1$, for
which $qd_1x_2x_3x_4$ is a non-square. In general, for any integer
$k$, there are  $O(Q^{1/2})$  integers $q\in [-Q, Q]$ for which
$kq$ is a square. Thus if we adjust the sum over $q$ to include all
integers with $X_1/2d_1<|q|\le X_1/d_1$, and then apply the 
Burgess bound, we find that
\[\sum_{q}\left(\frac{q}{n}\right)\ll (X_1/d_1)^{1/2}n^{1/5},\]
provided that $n$ is not a square. On the other hand, if $n$ is a
square, we have a trivial bound $O(X_1/d_1)$.  Thus
\begin{align*}
\sum_{\substack{n\le \hat{X}^{2/5}\\ \text{$n$ odd}}}  
\frac{1}{n^{\sigma}}  
\left|\sum_{q}
\left(\frac{q}{n}\right)\right|&\ll 
\sum_{n\le \hat{X}^{2/5}}n^{1/5-\sigma}(X_1/d_1)^{1/2}+
\sum_{m\le \hat{X}^{1/5}}m^{-2\sigma}(X_1/d_1)\\
&\ll  \hat{X}^{2/25}(X_1/d_1)^{1/2}+X_1/d_1.
\end{align*}
Since $X_1/d_1\ge(\hat{X}/\hat{d})^{1/4}$ we will have
\[\hat{X}^{2/25}(X_1/d_1)^{1/2}\le
\hat{X}^{2/25}(X_1/d_1)(\hat{X}/\hat{d})^{-1/8}\le
(X_1/d_1)\hat{d}^{1/8}.\]
When we sum over $x_2,x_3,x_4$ we now find that
\[\sum_{\x\in S}\sum_{n\le\hat{X}^{2/5}}\frac{\chi(n)}{n^{\sigma}}\ll
\frac{\hat{X}}{\hat{d}}\; \hat{d}^{1/8}=\hat X \hat d^{-7/8}.\]
This completes the proof of 
\eqref{lem:L} and so the proof of 
the lemma.
\end{proof}

To finish the proof of Lemma \ref{lem:stage1} we proceed to
show how to remove the condition $\Dbad(\x)\le Y^{1/20}$ 
from Lemma \ref{lem:roger}. 
It follows from Lemma  \ref{lem:bad_coeffs} that 
\[\sum_{\substack{\x\in (\ZZ_{\neq 0})^4\\ |\x|\leq
    X\\ \Dbad(\x)>Y^{1/20}}} M_1(\x;Y) \ll_\ve  
(XY)^{\ve}\left\{X^3 Y^{2-1/480} +X^4Y \right\}.\]
When $\ve=1/800$ we have  
\[(XY)^{\ve}X^3Y^{2-1/480}=(X^3Y^2)^{1-1/1600}(X^5Y^{2/3})^{1/1600}\le 
X^3Y^2+X^5Y^{2/3}\]
  and 
  \[(XY)^{\ve}X^4Y\le X^{13/3}Y^{10/9}=(X^3Y^2)^{1/3}(X^5Y^{2/3})^{2/3}
    \le  X^3Y^2+X^5Y^{2/3}.\] 
This completes the proof of the first part of
Lemma \ref{lem:stage1}.

Next, if $|\x|\le X$ with $\x\in(\ZZ_{\neq 0})^4$ 
and $\Delta(\x)=k^2$ say, then $0<|k|\le X^2$,
and each such $k$ corresponds to at most $8\tau_4(k^2)\ll X^{\ve}$
vectors $\x$. For each such $\x$
we use the bound $M_1(\x;Y)=O_\ve (Y^{2+\ve})$, which follows from 
Lemma~\ref{ann}. Hence 
$$
\sum_{\substack{\x\in (\ZZ_{\neq 0})^4\\ |\x|\leq
    X\\ \Delta(\x)=\square}} M_1(\x;Y) \ll_\ve  X^2Y^2 (XY)^\ve, 
$$
for any $\ve>0$, giving us the second part of Lemma \ref{lem:stage1}.

Finally, with $S_k(\alpha;X)$ as in (\ref{eq:Salpha}), we have
\[\sum_{\substack{\x\in (\ZZ_{\not=0})^4\\ |\x|\leq X,\, |x_1|\le X_1}}
M_1(\x;Y)\le\int_0^1S_1(\alpha;X)^3S_1(\alpha;X_1)\d\alpha,\]
whence H\"{o}lder's inequality yields
\[\sum_{\substack{\x\in (\ZZ_{\not=0})^4\\ |\x|\leq X,\, |x_1|\le X_1}}
M_1(\x;Y)\le\left\{\int_0^1|S_1(\alpha;X)|^4\d\alpha\right\}^{3/4}  
\left\{\int_0^1|S_1(\alpha;X_1)|^4\d\alpha\right\}^{1/4}.\]
Appealing to \eqref{eq:orkney}, this is 
\begin{align*}
  &\ll_\ve (XY)^\ve \left(X^3Y^2+X^4Y\right)^{3/4}
  \left(X_1^3Y^2+X_1^4Y\right)^{1/4}\\
&\ll_\ve (XY)^\ve X_1^{3/4}X^{-3/4}\left(X^3Y^2+X^4Y\right),
\end{align*}
and the third part of Lemma \ref{lem:stage1} follows.

\section{An asymptotic formula using lattice point counting}
\label{s:lattice}

In this section we write
\begin{equation}\label{eq:deg1}
M_4(\y;R)=\#\left\{ \x\in \ZZ^4: |\x|\leq R, ~ F(\x;\y)=0  \right\}
\end{equation}
and prove an asymptotic formula for
\beql{sm}
\sum_{\substack{\y\in \ZZp^4\\ Y<|\y|\leq 2Y}}
M_4\left(\y;(B/|\y|^2)^{1/3}\right)=N_1(B;Y),
\eeq
say.

\begin{theorem}\label{t:asymplatt}
  Let $Y\ge \tfrac12$. Then 
\[N_1(B;Y)=B\sum_{\substack{\y\in \ZZp^4\\ Y<|\y|\leq 2Y}}
\frac{\rho_\infty(\y)}{|\y|^2}+O(B^{2/3}Y^{4/3})+O(BY^{-1/3})+O(Y^4),\]
where
\begin{equation}\label{eq:def-rho}
  \rho_\infty(\y)=
  \int_{-\infty}^\infty \int_{[-1,1]^4} e(-\theta F(\x;\y)) \d \x \d \theta.
 \end{equation}
\end{theorem}

We note that if $\y$ has at least two
non-zero components, then
\[\int_{[-1,1]^4} e(-\theta F(\x;\y))\d\x  
=\prod_{j=1}^4\frac{\sin(2\pi\theta y_j^2)}{\pi\theta y_j^2} 
\ll_{\y}(1+|\theta|)^{-2},\] 
so that the outer integral in (\ref{eq:def-rho}) is absolutely
convergent. On the other hand, if $\y=(1,0,0,0)$, for example, then
\[\int_{[-1,1]^4} e(-\theta F(\x;\y))\d\x  
=8\frac{\sin(2\pi\theta)}{\pi\theta},\]
and the outer integral is conditionally convergent, with value 8.

We begin the proof by estimating $M_4(\y;R)$ for an individual vector
$\y$, as follows.
\begin{lemma}\label{lem:numbers}
Let $\y\in \ZZp^4$ and put 
$\mathsf{d}(\y)=\sqrt{y_1^4+\dots+y_4^4}.$
Let $V(\y)$ be the volume of the intersection of the cube $[-1,1]^4$
  with the hyperplane
$$
\{\x\in\RR^4:~F(\x;\y)=0\}.
$$
Then there exists a vector $\x_1=\x_1(\y)\in \ZZp^4$ satisfying
\beql{c}
0<|\x_1|\ll |\y|^{2/3}\;\;\;\mbox{and}\;\;\; F(\x_1;\y)=0,
\eeq
such that 
\[M_4(\y;R)=\frac{ V(\y)}{\mathsf{d}(\y)}R^3
+O\left(\frac{R^2}{|\x_1|^2}\right)+O(1). \]
\end{lemma}

If $(y_1^2,\ldots,y_4^2)=(z_1,\ldots,z_4)$ we may apply a rotation 
$\mathbf{R}\in\mathrm{SO}_4(\RR)$ to $\x$ and $\z$ so as to move $\z$ to
$\mathbf{R} \z=(0,0,0,z_4')$, say. Then $V(\y)$ will be the 3-dimensional volume of the region
$\{\mathbf{t}\in \mathbf{R}[-1,1]^4: t_4=0\}$.

\begin{proof}[Proof of Lemma \ref{lem:numbers}]
When $\y$ is primitive the
function $M_4(\y;R)$ counts vectors $\x\in\ZZ^4$ from a
3-dimensional lattice $\mathsf{\Lambda}$ of determinant $\mathsf{d}(\y)$, 
as in \cite[Lemma 1(i)]{annal}, for example. We now claim that
$$
M_4(\y;R)=
\frac{V(\y)}{\mathsf{d}(\y)}R^3
+O\left(\frac{R^2}{\lambda_1\lambda_2}+
\frac{R}{\lambda_1}+1\right),
$$
where the implied constant is absolute
and $\lambda_1\leq \lambda_2\leq \lambda_3$ are the successive minima of
$\mathsf{\Lambda}$.  If we had been using the $L^2$-norm in place of
the $L^{\infty}$-norm this would have followed directly from Schmidt
\cite[Lemma~2]{schmidt}. One should note here that, in Schmidt's notation,
$\Lambda^{k(|-i)}$ contains the vectors
$\mathfrak{g}_1,\ldots,\mathfrak{g}_{k-i}$, and has successive minima
$\lambda_1,\ldots,\lambda_{k-i}$, so that
$d(\Lambda^{k(|-i)})\gg_k \lambda_1\ldots\lambda_{k-i}$. To handle the
$L^{\infty}$-norm one needs only trivial modifications to Schmidt's
argument, which we leave to the reader.
To complete the proof of Lemma \ref{lem:numbers}
we note that $R^2/(\lambda_1\lambda_2)\leq (R/\lambda_1)^2$
and $R/\lambda_1\leq \max((R/\lambda_1)^2, 1)$.
Moreover, by the definition of the successive minima 
the lattice $\mathsf{\Lambda}$ contains a vector of length $\lambda_1$.
Writing $\x_1$ for
this vector we see that $\x_1$ will be primitive, and 
the lemma follows, since
$\lambda_1\le(\lambda_1\lambda_2\lambda_3)^{1/3}\ll
\mathsf{d}(\y)^{1/3}
\ll|\y|^{2/3}.$
\end{proof}

We turn now to the proof of Theorem \ref{t:asymplatt}. In our argument
certain ``bad'' vectors $\y$
will have to be dealt with separately. We denote the set of these by
$\mathcal{B}$, and write $\mathcal{G}$ for the
remaining set of good vectors.  The
definition of the set $\mathcal{B}$ will be given below in \eqref{eq:B-bad}, 
since it is
hard to motivate at this stage.

For the bad vectors we note that
\[
B\frac{V(\y)}{|\y|^2\mathsf{d}(\y)}
\ll
BY^{-4},\]
whence
\[M_4\left(\y;(B/|\y|^2)^{1/3}\right)=B\frac{V(\y)}{|\y|^2\mathsf{d}(\y)}
+O(BY^{-4})+O\left(M_4\left(\y;(B/|\y|^2)^{1/3}\right)\right)\]
when $\y\in\mathcal{B}$.  
It therefore follows from (\ref{sm}) along with Lemma \ref{lem:numbers} that
\beql{n1b}
\begin{split}
N_1(B;Y)= B\sum_{\substack{\y\in \ZZp^4\\ Y<|\y|\leq 2Y}}
\frac{V(\y)}{|\y|^2\mathsf{d}(\y)}&+O(BY^{-4}\card \mathcal{B})+O(\Sigma_1)\\
&
+O(B^{2/3}Y^{-4/3}\Sigma_2)+O(Y^4),
\end{split}
\eeq
where
\[\Sigma_1=\sum_{\y\in\mathcal{B}}M_4\left(\y;(B/|\y|^2)^{1/3}\right),\]
and
\[\Sigma_2=\sum_{\y\in\mathcal{G}}|\x_1(\y)|^{-2}.\]

We begin by discussing $\Sigma_2$, since this motivates our choice of
the sets $\mathcal{G}$ and $\mathcal{B}$.  We define
\[E(Y)= \sum_{\y\in \mathcal{G}}
\sum_{\substack{\x\in \ZZp^4\\ 0<|\x|\le c|\y|^{2/3}\\ F(\x;\y)=0}}
\frac{1}{|\x|^2},\]
where $c$ is the implied constant in (\ref{c}).
We shall prove the following bound for this sum,
which shows that 
the term $\Sigma_2$ in (\ref{n1b}) makes a satisfactory
contribution in Theorem~\ref{t:asymplatt}.

\begin{lemma}\label{lem:E(Y)}
We have $E(Y)\ll  Y^{8/3}$
for any $Y\ge 1$.
\end{lemma}

\begin{proof}
Our strategy for estimating $E(Y)$ will be to sort the inner sum into
dyadic intervals for $|\x|$.   
When all the components of $\x$ are non-zero we shall be able to
invoke the second part of Lemma \ref{lem:stage1}, and when exactly three
of the components of $\x$ are non-zero we will use
Lemma \ref{l3}. Thus the remaining
vectors $\x$ are those with at most two non-zero components, and these
will correspond to $\y$ being in the bad set $\mathcal{B}$, which we now
proceed to describe.

If $\x$ has exactly one
non-zero component the equation $F(\x;\y)=0$ forces
the corresponding component of $\y$ to vanish. We therefore include
vectors $\y$ with $\prod y_i=0$ in the bad set $\mathcal{B}$.

Suppose on the other hand that exactly two components of $\x$ vanish,
say $x_3=x_4=0$, and that $\prod y_i\not=0$. 
Then $(x_1,x_2)=1$, since $\x$ is
primitive.  Moreover $x_1y_1^2+x_2y_2^2=0$. If we write $h=(y_1,y_2)$ then
we must have $x_2=\pm(y_1/h)^2$ and $x_1=\mp(y_2/h)^2$.  It follows that
$|y_1|/h\le c^{1/2}|\y|^{1/3}$,
and similarly $|y_2|/h\le c^{1/2}|\y|^{1/3}$.  We
then say that a vector $\y$ is ``bad'' if either $\prod y_i=0$ or if
there are two components, $y_1$
and $y_2$ say, such that
$|y_1|, |y_2|\le c^{1/2}(y_1,y_2)|\y|^{1/3}$. Here $c$ is the implied
constant in (\ref{c}).  We
now define 
\begin{equation}\label{eq:B-bad}
\mathcal{B}=\left\{\y\in \ZZp^4: 
\begin{array}{l}
Y<|\y|\leq 2Y, \text{ either $ \prod y_i=0$ or else } \\
\text{
$|y_i|, |y_j|\le c^{1/2}(y_i,y_j)|\y|^{1/3}$ for some $i\neq j$}
\end{array}
\right\}.
\end{equation}
Similarly we write $\mathcal{G}$ for the complement of 
$\mathcal{B}$ in the set of $\y\in \ZZp^4$
with $Y<|\y|\leq 2Y$.   Thus if $\y$ is in the
complementary set $\mathcal{G}$, any corresponding vector $\x$ has 
at most one zero entry.

We are now ready to estimate $E(Y)$.
Let $S(L,\mathcal{G})$ be the number of pairs $\x,\y$ that arise for which
  $L/2<|\x|\le L$.  Then
\[E(Y) \ll  \sum_L L^{-2}S(L,\mathcal{G}),\]
the sum over $L$ being for powers of 2 only, with
$L\ll Y^{2/3}$. Our definitions ensure that 
$$
S(L,\mathcal{G})\ll M_3(L,2Y)+YT(L,2Y,L,2Y), 
$$
in the notation of Lemmas \ref{lem:stage1} and \ref{l3}, which then yield
\[S(L,\mathcal{G})\ll_\ve L^3Y^2 +L^5Y^{2/3}+L^{2+\ve}Y^{2+\ve},\]
for any $\ve>0$.  For $L\ll Y^{2/3}$ and $\ve=3/10$ this is 
$\ll L^3Y^2+L^2Y^{5/2}$, whence
\[E(Y)\ll  \sum_{2^i\ll Y^{2/3}}\left(2^i Y^2 +Y^{5/2}\right)  
\ll Y^{8/3}+Y^{5/2}\log Y,\]
which is satisfactory for Lemma \ref{lem:E(Y)}. 
\end{proof}

We next estimate $\card\mathcal{B}$. There are $O(Y^3)$ vectors $\y$ with
$\prod y_i=0$.  For the
remaining bad vectors, if we have $y_1=hz_1$, for example, then
$y_2=hz_2$ with $z_1,z_2$ coprime and
\[|z_1|,|z_2|\le c^{1/2}|\y|^{1/3}\ll Y^{1/3}.\]
There are $O(Y)$ choices for $h$ and $O(Y^{1/3})$ choices for
each of $z_1$ and $z_2$, and since there are $O(Y^2)$ possible values for
$y_3$ and $y_4$ we see that there are $O(Y^{11/3})$ options for
$\y$. Thus $\card\mathcal{B}\ll Y^{11/3}$, so that the corresponding
term in (\ref{n1b}) is satisfactory for Theorem \ref{t:asymplatt}.

It remains to consider $\Sigma_1$. We begin by disposing of the
contribution from solutions $\x,\y$ with $\prod x_iy_i=0$.  We apply Lemma
\ref{lem:xz} with $X\ll B^{1/3}Y^{-2/3}$ obtaining a bound 
$O_{\ve}(BY^{\ve-1}+Y^4)$.  The corresponding contribution to (\ref{n1b})
will turn out to be satisfactory for our purposes, as we shall see shortly. 
In what follows we may assume $\prod x_iy_i\neq 0$.

We now focus on terms for which 
$y_1=hz_1$ and $y_2=hz_2$ where $z_1$ and $z_2$ are coprime and
$0<|z_1|,|z_2|\le c^{1/2}|\y|^{1/3}$, so that
\beql{eg}
(x_1z_1^2+x_2z_2^2)h^2+x_3y_3^2+x_4y_4^2=0
\eeq
with non-zero integer variables.  We set
\[X=(BY^{-2})^{1/3}\;\;\;\mbox{and}\;\;\;Z=\max\{|z_1|,|z_2|\},\]
whence $1\le h\le 2Y/Z$.  When $x_1z_1^2+x_2z_2^2=0$ we have
$x_3y_3^2+x_4y_4^2=0$ as well. There are then $O(Y/Z)$ choices for
$h$, while Lemma \ref{lem:2} shows that there are $O(XZ)$ values for
$x_1,x_2,z_1,z_2$ and $O(XY)$ possibilities for $x_3,x_4,y_3,y_4$.
The case $x_1z_1^2+x_2z_2^2=0$ therefore contributes
$O(X^2Y^2)=O(B^{2/3}Y^{2/3})$ to
$\Sigma_1$. We shall see in a moment that this makes a suitably 
small contribution to (\ref{n1b}). 

We count the remaining solutions according
to the values taken by $z_1$ and $z_2$.  It will be convenient in what
follows to write $N(Y;z_1,z_2)$ for the number of solutions
to the equation (\ref{eg})
in non-zero integers $x_1,\ldots,x_4,h,y_3,y_4$ with
$x_1z_1^2+x_2z_2^2\not=0$ and
\[|\x|\le X, ~1\le h\le 2Y/Z \quad \text{ and }
\quad 
|y_3|,|y_4|\le 2Y.\] 
It follows from our analysis thus far that 
\beql{NZs}
\Sigma_1\ll_{\ve}BY^{\ve-1}+Y^4+B^{2/3}Y^{2/3}+ 
\sum_{0<|z_1|,|z_2|\ll Y^{1/3}}N(Y;z_1,z_2).
\eeq

We put  $t=x_1z_1^2+x_2z_2^2$, which is assumed to be
non-zero.  For a given non-zero $t$ (and fixed $z_1,z_2$)
the number of $x_1,x_2$
such that $x_1z_1^2+x_2z_2^2=t$ is $O(1+X/Z^2)$. Moreover the equation
$th^2+x_3y_3^2+x_4y_4^2=0$ has at most $T(X,2Y,2XZ^2,2Y/Z)$ solutions,
in the notation of Lemma \ref{l3}, which then shows that
\[N(Y;z_1,z_2) \ll_{\ve} (1+XZ^{-2})X^2Y^{1+\ve}Z\]
for any fixed $\ve>0$.
We insert this estimate into (\ref{NZs}) and find that
\begin{align*}
  \Sigma_1 
&\ll_{\ve}BY^{\ve-1}+Y^4+B^{2/3}Y^{2/3}+(Y^{2/3}+X)X^2Y^{4/3+\ve}\\
&\ll_{\ve}Y^4+B^{2/3}Y^{2/3+\ve}+BY^{\ve-2/3}.
\end{align*}
Taking $\ve=1/3$ now gives us a suitable contribution to (\ref{n1b}).

Returning to \eqref{n1b}, 
in order to  complete the proof of Theorem \ref{t:asymplatt} it
remains to show that 
\begin{equation}\label{eq:2constants}
\frac{V(\y)}{\mathsf{d}(\y)}=\rho_\infty(\y),
\end{equation}
for any non-zero vector  $\y\in \RR^4$, 
where $\rho_\infty(\y)$ is
defined in  \eqref{eq:def-rho}. It will be convenient to put
$y_i^2=\mathsf{d}(\y)w_i$ for $1\leq i\leq 4$. Then if
$\|\cdot\|_2$ denotes 
the 
$L^2$-norm,  it follows that $\|\w\|_2=1$.  Moreover, 
in this notation we have
\[\rho_{\infty}(\y)\mathsf{d}(\y)=
\int_{-\infty}^\infty \int_{[-1,1]^4} e(-\theta \w.\x) \d \x \d \theta.\]
As already noted following the statement of Theorem \ref{t:asymplatt},
the repeated integral is $8$ if $\w$ has a single non-zero component.
This suffices for (\ref{eq:2constants}), since one easily sees that
$V(\w)=8$ in this case.

On the other hand, if $\w$ has at least two non-zero
components then, as remarked earlier, the inner integral is
$O((1+|\theta|^2)^{-1})$, so
that the double integral is
\begin{equation}\label{R1}
\begin{split}
\lim_{\delta\downarrow 0}
\int_{-\infty}^\infty
&\left(\frac{\sin(\pi\delta\theta)}{\pi\delta\theta}\right)^2
\int_{[-1,1]^4} e(-\theta \w.\x) \d \x \d \theta \\
  &=\lim_{\delta\downarrow 0}\int_{[-1,1]^4}\int_{-\infty}^\infty
\left(\frac{\sin(\pi\delta\theta)}{\pi\delta\theta}\right)^2
e(-\theta \w.\x)\d\theta\d\x. 
\end{split}
\end{equation}
In general if
\begin{equation}\label{eq:Ku}
K(u;\delta)=\begin{cases}
\delta^{-2}(\delta-|u|), &\text{ if $|u|\leq \delta$,}\\
0, &\text{ if $|u|\geq \delta$,}
\end{cases}
\end{equation}
then
\beql{KFT}
K(u;\delta)=\int_{-\infty}^{\infty}e(\theta u)
\left(\frac{\sin(\pi\delta\theta)}{\pi\delta\theta}\right)^2\d\theta,
\eeq
so that the inner integral on the right of (\ref{R1}) is $K(\w.\x;\delta)$.

Since $\|\w\|_2=1$ there exists  
a $4\times 4$ orthogonal matrix 
$\mathbf{M}\in O_4(\RR)$ 
such that $\mathbf{M}\w=(1,0,0,0)$. Then
$\w.\x=(\mathbf{M}\w)^T \mathbf{M}\x$. Changing variables from $\x$ to 
$\z=\mathbf{M}\x$, so that $\z$ runs over the set
$Z=\mathbf{M}[-1,1]^4$, we see that
\[\rho_{\infty}(\y)\mathsf{d}(\y)=
\lim_{\delta\downarrow 0} \int_{Z} K(z_1;\delta)\d\z
=\meas\{\z\in Z: z_1=0\}=V(\w)=V(\y),\]
as required.
This concludes the proof of Theorem \ref{t:asymplatt}.

\section{Counting points on quadrics}\label{CPOQ}

In this section we will establish an asymptotic formula for the smoothly
weighted counting function 
\[N_w(P)=\sum_{\substack{\x\in \ZZ^4\\ F(\x)=0}} w(P^{-1}\x),\]
as $P\to \infty$, where $F(\x)$ is the non-singular integral
diagonal quadratic form
\[F(\x)=A_1x_1^2+A_2x_2^2+A_3x_3^2+A_4x_4^2. \]
Here $w:\RR^4\to \RR_{\geq 0}$ is a fixed infinitely
differentiable weight function of compact support, which vanishes in some
neighbourhood of the origin. Our goal is to establish an asymptotic
formula even when the coefficients $A_i$ are of size a small power of
$P$, and it will be crucial to our success that the size we are able
to handle is sufficiently large.

Our asymptotic formula for $N_w(P)$ 
is only valid for suitable weights $w$ and ``generic'' choices of the
coefficients $A_i$. To specify the necessary conditions we define
\[\|F\|=\max_{1\le i\le 4}|A_i|,  \quad 
\Delta_F=A_1A_2A_3A_4 \big(\neq 0\big),\]
and
\beql{eq:bad-Delta}
\Dbad=\prod_{\substack{p^e\| \Delta_F\\ e\ge 2}}p^e.
\eeq
We then require that
\beql{eq:assume1-w}
w(\x)=0\;\;\;\mbox{for}\;\;\; |\x|\leq\eta,
\eeq
that
\begin{equation}\label{eq:assume1-x}
\|F\|^{1-\eta}\leq |A_i|\big(\leq \|F\|\big), \quad \text{for  $1\leq i\leq 4$},
\end{equation}
and  that
\begin{equation}\label{eq:assume2-x}
\Dbad \leq \|F\|^{\eta},
\end{equation}
for a small parameter
$\eta\in(0,\tfrac{1}{100})$ at our disposal. 
The first two conditions imply that
\beql{NL}
|\nabla F(\x)|\gg_{\eta}\|F\|^{1-\eta}\;\;\mbox{for}\;\; w(\x)\not=0,
\eeq
while the last
condition implies in particular that $\Delta_F\not=\square$ when $\|F\|>1$.

Our asymptotic formula
involves the ``singular integral'', defined to be  
\begin{equation}\label{eq:sigma-infty-omega}
  \sigma_\infty(w;F)=\int_{-\infty}^{\infty}\int_{\RR^4}w(\x)
  e(-\theta F(\x))\d\x\d\theta,
\end{equation}
and the ``singular series''
\begin{equation}\label{eq:S}
\mathfrak{S}(F)=\prod_p \lim_{r\to\infty} p^{-3r}\#\{\x\in
(\ZZ/p^r\ZZ)^4: F(\x)\equiv 0\bmod{p^r}\}. 
\end{equation}
We will see in Lemma \ref{SSB} that this is convergent whenever 
$\Delta_F\neq \square$. With this notation
our principal result in this section is the following.

\begin{theorem}\label{newth}
  When \eqref{eq:assume1-w}, \eqref{eq:assume1-x} and
  \eqref{eq:assume2-x} hold we have
  \[N_w(P)=\sigma_\infty(w;F)\mathfrak{S}(F)P^2
  +O_{w,\eta}(\|F\|^{-1/2+5\eta}P^{3/2}+\|F\|^{1/2+2\eta}P),\] 
  provided that $\|F\|\ge P^{\eta}$.
\end{theorem}

The main term here is typically of size around $P^2\|F\|^{-1}$, so
that we get an asymptotic formula when $P^{\eta}\le\|F\|\le P^{2/3+O(\eta)}$.
For comparison, Browning \cite[Prop.~2]{qupper} shows that
$$
  N_w(P) = \sigma_\infty(w;F)\mathfrak{S}(F)P^{2}+
  O_{w,\eta}\left(\|F\|^{3+9\eta}P^{3/2+\eta} \right),
$$
for a special choice of  weight function $w$, under the assumptions
that  $\Delta_F\neq \square$ and that
(\ref{eq:assume1-x}) and (\ref{eq:assume2-x}) hold.
Theorem \ref{newth} refines this result 
considerably  for forms  
whose discriminant is close to being square-free, provided that the
coefficients of $F$ are not too small compared with $P$. We should
emphasize that in both Theorem \ref{newth} and the result of Browning
\cite{qupper} the coefficients of $\eta$ are relatively unimportant.
In particular they have no significance for the current application.

The condition that $\|F\|\geq P^\eta$ is somewhat unnatural and
deserves further comment. Under this assumption together with
the hypotheses (\ref{eq:assume1-w}), (\ref{eq:assume1-x}) and
(\ref{eq:assume2-x})
we are able to eliminate certain
awkward terms that arise when we apply Poisson summation.  This is
explained further in \S \ref{ssEB}.
At this point it is crucial that the
quadratic form $F(\x)$ is diagonal.
We could remove the condition $\|F\|\ge P^\eta$ and handle
non-diagonal forms, but this would be at the expense
of a worse dependence on $\|F\|$. 

\subsection{Preliminaries}

Our proof of Theorem \ref{newth} uses 
the smooth $\delta$-function variant of the circle method introduced by  
Duke, Friedlander and Iwaniec \cite{DFI}, and later  developed by
Heath-Brown \cite[Thm.~1]{HB'}.  We proceed to review the technical
apparatus required.

For any $q \in \NN$, any $\c \in \ZZ^4$ and any $Q\geq 1$, we define
the complete exponential sum 
\begin{equation}\label{eq:SUM}
S_q(\c) = \sumstar_{\substack{a\bmod{q}}}~
\sum_{\b\bmod{q}}  e_q \left(aF(\ma{b})+\b.\c\right),
\end{equation}
and the oscillatory integral 
$$
I_q(\c)=
\int_{\RR^4}w(P^{-1}\x)h\left(\frac{q}{Q},\frac{F(\x)}{Q^2}\right)
e_q(-\c.\x)\d\x,  
$$
for a certain function 
 $h:(0,\infty)\times \RR \rightarrow \RR$ 
described in  \cite[\S 3]{HB'}.
We note here that $h(x,y)$ is independent of $F$ and $Q$ and is infinitely 
differentiable for $(x,y) \in (0,\infty)\times \RR$. 
Moreover $h(x,y)$ is non-zero only for $x \leq \max\{1,2|y|\}$. 
When $\tfrac12\le x\le 1$ and $|y|\le\tfrac14$, for example, 
$h(x,y)$ will be of exact order 1. 
It then follows from   \cite[Thm.~2]{HB'} that 
\begin{equation}\label{asym1} 
N_w(P) = \frac{c_Q  }{Q^{2}}\sum_{\c\in \ZZ^4} \sum_{q
=1}^\infty q^{-4}S_q(\c)I_q(\c).  
\end{equation}
where the constant $c_Q$ satisfies
$c_Q=1+O_N(Q^{-N})$, for any $N>0$. 

We shall take $Q=\sqrt{\|F\|P^2}$ in our work.  
In our proof of Theorem \ref{newth} we shall often encounter a small
positive parameter $\ve$, and for  the sake of convenience 
we shall allow it to take different values at different stages
of the argument, so
that $x^\ve\log x\ll x^\ve $, for example. 
All of our implied constants  are allowed to depend on the weight function
$w$ and on $\eta$ and $\ve$, but on nothing else unless specified. 
Ultimately we will
take $\ve$ to be fixed in terms of $\eta$ but much smaller than it,
so that the dependence on
$\ve$ will disappear.  As above we assume that $w$,
besides being infinitely differentiable and of compact support,
satisfies the condition \eqref{eq:assume1-w}.

We now wish to apply the bounds for the exponential integral $I_q(\c)$
that were derived in  \cite[\S\S 7--8]{HB'}. Unfortunately, the
implied constants in each  of these estimates is allowed to depend
implicitly  on the coefficients of $F$, a deficiency that we shall
need to remedy here.   

\begin{lemma}\label{lem:I-estimate}
Let $\c\in \ZZ^4$ be non-zero.
Then the following hold:
\begin{enumerate}
\item[(i)] For any $N\geq 0$ we have
\[I_q(\c)\ll_{N}  \frac{P^{5}}{q} \frac{\|F\|^{(N+1)/2}}{|\c|^{N}}.\]
\item[(ii)]
We have 
\[I_q(\c) \ll \frac{q\|F\|^{2}P^{3}}{|\Delta_F|^{1/2}|\c|}.\]
\end{enumerate}

\end{lemma}

\begin{proof}
To begin with we may write
\[I_q(\c)=P^4 \int_{\RR^4} w(\x) h(r,G(\x)) e\left(-\v.\x\right)\d\x,\]
where $r=q/Q$, $\v=q^{-1}P\c$ and $G=\|F\|^{-1}F$. In particular 
\[\frac{\partial^{k_1+\dots + k_4}}{\partial x_1^{k_1}\dots \partial
  x_4^{k_4}}G(\x)\ll_{k_1,\dots,k_4} 1, \]
for all $\x\in \supp(w)$ and $k_1,\dots,k_4\in \ZZ_{\geq 0}$.
The function $h(x,y)$  is non-zero only for $x\leq \max\{1,2|y|\}$
and satisfies $h(x,y)\ll x^{-1}$. In fact, for any $i,j\geq 0$ and any
$N\geq 0$,  we have  
\begin{equation}\label{eq:h-function}
 \frac{\partial^{i+j}h(x,y)}{\partial^i x\partial y^j} \ll_{i,j,N}
 x^{-1-i-j}\left(x^N+ 
 \min\{1,(x/|y|)^N\}\right).
 \end{equation}
These facts are all explained in \cite[Lemmas 4 and 5]{HB'}.
Repeated integration by parts now establishes part (i).

Turning to part (ii), we see that 
$I_q(\c)=P^4r^{-1}I(\v)$, in the notation of \cite[Lemma 14]{HB'}, with $f=rh$.
The argument in the proof of \cite[Lemma~17]{HB'} shows that there
exists a smooth weight function $w_1:\RR^4\to \RR_{\geq 0}$, such that  
$\supp(w_1)\subseteq \supp(w)$ 
and a function $p(t)\ll r $ such that 
$$
I(\v)=\int_{-\infty}^\infty p(t)\int_{\RR^4} w_1(\x)
e\left(tG(\x)-\v.\x\right) \d\x \d t. 
$$
We may analyse the inner integral $J$, say,  using the smoothly
weighted stationary phase bound worked  
out by Heath-Brown and Pierce in 
Lemmas 3.1 and 3.2 of 
 \cite{hb-pierce}.
Recall that $G(\x)=\|F\|^{-1}F(\x)\ll 1$ for all $\x\in \supp(w_1)$ and 
observe that 
 integration by parts gives
$$
 \int_{\RR^4} |\hat w_1(\y)|\d \y\ll
 \int_{\RR^4}\min\left\{1,|\y|^{-1}\right\}^5 \d \y\ll 1. 
$$
This shows that 
 $J=O_M(|\v|^{-M})$ for any $M\geq 0$ 
if $|\v|\gg |t|$, while we have 
$J=O(\|F\|^{2}|\Delta_F|^{-1/2}t^{-2})$ otherwise. 
Applying these bounds with $M=2$, 
and noting that $|\Delta_F|^{1/2}\leq \|F\|^2$, 
 the statement of part (ii) easily follows.
\end{proof}

The effect of part (i)  is that the sum over $\c$ in 
\eqref{asym1} can be truncated to $|\c|\ll \|F\|^{1/2}Q^\ve$
for any $\ve>0$, with negligible  error. The following estimate allows
us to work with $I_q(\c)$ when $\c=\0$. 

\begin{lemma}\label{lem:I0}
Assume that \eqref{eq:assume1-w} and \eqref{eq:assume1-x} hold. Then 
 we have 
\[q^k \frac{\partial^k I_q(\0)}{\partial q^k}\ll \|F\|^\eta P^4, \quad
\text{for $k\in \{0,1\}$}. \]
\end{lemma}

\begin{proof}
Let $k\in\{0,1\}$
and recall the notation $r=q/Q$. Then 
\begin{align*}
q^k \frac{\partial^k I_q(\0)}{\partial q^k}
&=r^kP^4 \int_{\RR^4} w(\x) 
\frac{\partial^k h(r,G(\x))}{\partial r^k}\d\x\\
&\ll r^{-1}P^4 \int_{\x\in \supp(w)} 
\left(r^2+\min\left\{1,\frac{r^2}{G(\x)^2}\right\}\right)
\d\x,
\end{align*}
on  taking $N=2$ in \eqref{eq:h-function}.
We now appeal to \eqref{NL}, which implies that 
$|\nabla G(\x)|\gg \|F\|^{-\eta}$ 
for all $\x\in \supp(w)$. Thus 
the measure of the set where $|G(\x)|\leq z$ is $O(z\|F\|^\eta
)$. The integral is therefore $O(r\|F\|^\eta )$, as in the proof of
\cite[Lemma 15]{HB'}, and the  lemma follows. 
\end{proof}

We conclude this section by considering the integral
\beql{J1d}
J(\theta;w)=\int_{\RR^4}w(\x)e(-\theta G(\x))\d\x.
\eeq
\begin{lemma}\label{J1L}
Under the assumptions \eqref{eq:assume1-w} and \eqref{eq:assume1-x} we
have
\[J(\theta;w)\ll_N |\theta|^{-N}\|F\|^{2N\eta}\]  
for any non-negative integer $N$.
\end{lemma}

\begin{proof}
To prove this we use the first derivative bound for
smooth exponential integrals, see Heath-Brown
\cite[Lemma~10]{HB'}.
First however we must reduce the support of the
weight function $w$ by using Lemma 2 of \cite{HB'}. This shows that if
$0<\delta<1$ then there is a smooth function $w_{\delta}$ of compact
support such that
\[w(\x)=\delta^{-4}\int_{\RR^4} w_{\delta}(\delta^{-1}(\x-\y),\y)\d\y.\]
Thus there is some vector $\y=\y(\delta)$ such that
\[J(\theta;w)\ll
\delta^{-4}\left|\int_{\RR^4}w_{\delta}(\delta^{-1}(\x-\y),\y)
e(-\theta G(\x))\d\x\right|.\]
We will choose $\delta=\|F\|^{-1}\min|A_i|$. Thus
$\delta\geq \|F\|^{-\eta}$, by \eqref{eq:assume1-x}.  We now see that
\beql{J1e}
J(\theta;w)\ll\left|\int_{\RR^4}w_*(\mathbf{u})
e(-\theta G_*(\u))\d\u\right|,
\eeq
where we have set $w_*(\u)=w_{\delta}(\u,\y)$ and
$G_*(\u)=G(\y+\delta\u)$. According to \cite[Lemma 2]{HB'}, if
$w_*(\u)\not=0$ then $w(\y+\delta\u)\not=0$, so that
$|\y+\delta\u|\ge\eta$, by (\ref{eq:assume1-w}).  It follows that
$|\nabla G_*(\u)|\gg_{\eta}\delta^2$ on the support of $w_*$.
Moreover, the second order derivatives of $G_*$ are $O(\delta^2)$ and
the higher derivatives vanish.  We then see from \cite[Lemma 10]{HB'} that
\[\int_{\RR^4}w_*(\mathbf{u})e(-\theta G_*(\u))\d\u\ll_{N,w,\eta}
|\theta|^{-N}\delta^{-2N}\]
for any non-negative integer $N$. The reader should note that the
implied constant is independent of $\y$ and $\delta$ by the technical
properties of $w_{\delta}$ described in \cite[Lemma 2]{HB'}.
The statement now follows from \eqref{J1e}.
\end{proof}

\subsection{The exponential sums}

In this section we summarise what we will need to know about the 
 exponential sums $S_q(\mathbf{c})$ defined in 
\eqref{eq:SUM}.  As proved in 
\cite[Lemma 23]{HB'}, these satisfy the 
multiplicativity property $S_{q_1q_2}(\c)=S_{q_1}(\c)S_{q_2}(\c)$ for
any coprime integers $q_1,q_2$.   
We begin by establishing the following basic estimate.

\begin{lemma}\label{lem:Sq-basic}
For any $q\in \NN$ we have 
$$
S_q(\c)\ll q^3 \prod_{1\leq i\leq 4} (q,A_i,c_i)^{1/2}.
$$
\end{lemma}
\begin{proof}
To begin with, for any prime power $p^r$ we have 
\begin{equation}\label{eq:busy}
S_{p^r}(\c)=\sumstar_{a\bmod{p^r}}
\prod_{1\leq i\leq 4} G(aA_i,c_i;p^r),
\end{equation}
where
\begin{equation}\label{eq:gauss}
G(b,c;q)=
\sum_{x\bmod{q}}e_{q}\left(b x^2+cx\right)
\end{equation}
denotes the generalised Gauss sum for any $q\in \NN$ and $b,c\in \ZZ$.
For $q=p^r$ we put
 $\beta=v_p(b)$.
Breaking into residue classes modulo $p^{r-\min\{\beta,r\}}$, it is
easy to see that $G(b,c;p^r)=0$ unless $p^{\min\{\beta,r\}}\mid c$. 
We  conclude that if $S_{p^r}(\c)\neq 0$ then we must have 
$\min\{v_p(A_i),r\}\leq v_p(c_i)$, for $1\leq i\leq 4$.

Next, 
on inspecting the proof of \cite[Lemma 25]{HB'}, we find that 
\begin{align*}
S_{p^r}(\mathbf{c}) 
&\leq p^{3r} \#\{\y\bmod{p^r}: p^r\mid 2A_i y_i \text{ for $1\leq
  i\leq 4$}\}^{1/2}\\ 
&\leq p^{3r} 
\prod_{1\leq i\leq 4}
p^{(\min\{v_p(A_i),r\}+v_p(2))/2}\\ 
&= p^{3r+2v_p(2)} 
\prod_{1\leq i\leq 4}
p^{\min\{v_p(A_i),r,v_p(c_i)\}/2}. 
\end{align*}
The statement of the lemma now follows from multiplicativity.
\end{proof}

Our next result relies on an explicit evaluation of the Gauss sum
\eqref{eq:gauss}. 
Let $b,c\in \ZZ$ and let $q$ be an odd integer. Then according to
\cite[p.~66]{Iwaniec}, we have 
\begin{equation}\label{eq:gauss1}
G(b,c;q)
=e_q\left(-\bar{4b}c^2\right) \left(\frac{b}{q}\right) \delta_q \sqrt{q},
\end{equation}
provided that $(b,q)=1$. Here $\bar{4b}$ denotes the multiplicative
inverse of $4b$ modulo $q$,  and  
\[\delta_q=\begin{cases}
1 & \text{if $q\equiv 1\bmod{4}$,}\\
i & \text{if $q\equiv 3\bmod{4}$.}
\end{cases}\]
Define the dual form
\begin{equation}\label{eq:F*}
F^*(\y)=
A_2A_3A_4y_1^2+
A_1A_3A_4y_2^2+
A_1A_2A_4y_3^2+
A_1A_2A_3y_4^2.
\end{equation}
The following result follows on inserting \eqref{eq:gauss1} into 
\eqref{eq:busy}.

\begin{lemma}\label{lem:good_primes}
Let $p\nmid 2\Delta_F$ be a prime. Then 
$$
S_{p^r}(\c)=\left(\frac{\Delta_F}{p^r}\right) p^{2r} c_{p^r}(F^*(\c)),
$$
where $c_{p^r}(N)$ is the Ramanujan sum.
\end{lemma}

By standard  properties of the Ramanujan sum,  this result implies that 
$S_{p^r}(\c)=0$ unless $p\mid F^*(\c)$, whenever 
$p\nmid 2\Delta_F$ and $r\geq 2$.

We proceed by using our work so far to study the asymptotic behaviour
of the sum 
\begin{equation}\label{eq:Sigma}
\Sigma(x;\c)=\sum_{q\leq x} q^{-3}S_q(\c),
\end{equation}
for
suitable  vectors  $\c$.

\begin{lemma}\label{lem:average1}
Let $\ve>0$ and 
assume that $F^*(\c)\neq 0$.
Then 
\[|\Sigma(x;\c)|\leq \sum_{q\leq x} q^{-3}|S_q(\c)|
\ll_\ve x^{\ve} |\c|^\ve \|F\|^\ve 
\prod_{1\leq i\leq 4} (A_i,c_i)^{1/2}.\]
\end{lemma}

\begin{proof}
Define the non-zero integer  $N=2\Delta_F F^*(\c)$.
To handle $\Sigma(x;\c)$  we sum trivially over $q$, finding that
$$
\sum_{q\leq x }q^{-3}|S_q(\c)|
=
\sum_{\substack{q_2\leq x\\ q_2\mid N^\infty}}q_2^{-3}|S_{q_2}(\c)|
\sum_{\substack{q_1\leq x/q_2\\ (q_1,N)=1}}q_1^{-3}|S_{q_1}(\c)|. 
$$
It follows from Lemma \ref{lem:good_primes} that the inner sum is
restricted to square-free integers $q_1$ and  that 
$|S_{q_1}(\c)|\leq q_1^2$.  Lemma \ref{lem:Sq-basic}
yields
\[
\sum_{q\leq x}q^{-3}|S_q(\c)|
\ll \log x
\sum_{\substack{q_2\leq x\\ q_2\mid N^\infty}}\frac{|S_{q_2}(\c)|}{q_2^3}
\ll \log x
\sum_{\substack{q_2\leq x\\ q_2\mid N^\infty}}
\prod_{1\leq i\leq 4} (A_i,c_i)^{1/2}.
\]
The statement of the lemma follows on noting that there are
$O_\ve(N^\ve x^\ve)$ values of $q_2$ that contribute to the remaining sum. 
\end{proof}

We shall also need to study $\Sigma(x;\0)$.
First, we recall the definition \eqref{eq:bad-Delta} of $\Dbad$ and
establish the following result.

\begin{lemma}\label{lem:zero-c}
Let $r\geq 1$ and  $p\mid \Delta_F$, with  $p\nmid  2\Dbad$. Then
$S_{p^r}(\0)=0$. 
\end{lemma}

\begin{proof}
We return to \eqref{eq:busy} with $\c=\0$. The assumption $p\nmid
2\Dbad$ implies that  $v_p(\Delta_F)=1$. We suppose without loss of
generality that  
$v_p(A_4)=1$ and $p\nmid A_1A_2A_3$.
We may evaluate 
$G(aA_i,0;p^r)$ using \eqref{eq:gauss1} for $1\leq i\leq 3$. Next, 
on writing $A_4'=A_4/p$,  it is easy to see that 
$$
G(aA_4,0;p^{r})=
\sum_{x\bmod{p^r}}e_{p^{r-1}}\left(aA_4' x^2\right)=
\begin{cases}
p &\text{ if $r=1$,}\\
(\frac{aA_4'}{p^{r-1}}) 
\delta_{p^{r-1}}p^{\frac{r+1}{2}} &\text{ if $r\geq 2$.}
\end{cases}
$$
Hence $S_{p^r}(\0)=0$  for $r\geq 1$, since 
the $a$-sum   is 
$
\sum_{a\bmod{p^r}}^* (\frac{a}{p})=
0.
$
\end{proof}

We now have everything in place to study 
 $\Sigma(x;\0)$.  This time we shall sum non-trivially over
$q$ using the Burgess bound for short character sums. 

\begin{lemma}\label{lem:average2}
Let $\ve>0$ and 
assume that  $\Delta_F\neq \square$.
Then 
$$
 \Sigma(x;\0) \ll_\ve   |\Delta_F| ^{3/16+\ve} 
 \Dbad^{3/8}~x^{1/2+\ve}.
$$
\end{lemma}

\begin{proof}
Define the non-zero integer  $N=2\Dbad$. 
 We have 
\begin{align*}
\Sigma(x;\0)
&=
\sum_{\substack{q_2\leq x\\ q_2\mid N^\infty}}q_2^{-3}S_{q_2}(\0)
\sum_{\substack{q_1\leq x/q_2\\ (q_1,N)=1}}q_1^{-3}S_{q_1}(\0). 
\end{align*}
It follows from Lemma~\ref{lem:zero-c} that  the inner sum is actually
only over $q_1$ which are coprime to $2\Delta_F$. 
Setting $A=2\Delta_F$ for convenience,
it follows from Lemma~\ref{lem:good_primes} that 
\begin{align*}
\sum_{\substack{q\leq X\\ (q,A)=1}}q^{-3}S_{q}(\0)=
\sum_{\substack{q\leq X\\ (q,A)=1}}\left(\frac{\Delta_F}{q}\right) \phi^*(q),
\end{align*}
where $\phi^*=1*h$, with $h(d)=\mu(d)/d$. Opening up $\phi^*$ and
inverting the order of summation, we conclude that  
$$
\sum_{\substack{q\leq X\\ (q,A)=1}}q^{-3}S_{q}(\0)=
\sum_{\substack{u\leq X\\ (u,A)=1}} \frac{\mu(u)}{u}
\left(\frac{\Delta_F}{u}\right)
\sum_{\substack{v\leq X/u\\ (v,A)=1}}\left(\frac{\Delta_F}{v}\right).
$$
According to Burgess \cite{burgess1,burgess2}, the inner sum is
$O_\ve(|\Delta_F|^{3/16+\ve} (X/u)^{1/2})$, for any $\ve>0$.
But then we have 
$$
\sum_{\substack{q\leq X\\ (q,A)=1}}q^{-3}S_{q}(\0)
\ll_\ve |\Delta_F|^{3/16+\ve}
X^{1/2}
\sum_{\substack{u\leq X}} \frac{1}{u^{3/2}} 
\ll_\ve |\Delta_F|^{3/16+\ve}
X^{1/2}.
$$
Applying this with $X=x/q_2$ and 
returning to the start of the argument, we may now deduce from 
Lemma 
\ref{lem:Sq-basic}
that 
\begin{align*}
\Sigma(x;\0)
&\ll_\ve
 |\Delta_F| ^{3/16+\ve} x^{1/2}
\sum_{\substack{q_2\leq x\\ q_2 \mid \Dbad^\infty}}
\frac{\prod_{1\leq i\leq 4} (q_2,A_i)^{1/2}}{q_2^{1/2}}
\ll_\ve 
 |\Delta_F| ^{3/16+\ve} 
 \Dbad^{3/8}
 x^{1/2+\ve},
\end{align*}
since there are $O_\ve(|\Delta_F|^\ve x^\ve)$ choices for $q_2$ in this sum.
This completes the proof of the lemma.
\end{proof}

We end this section by considering the singular series (\ref{eq:S}),
which we may write as
\[\mathfrak{S}(F)=\prod_p\sum_{r=0}^{\infty}p^{-4r}S_{p^r}(\0).\]
We prove the following upper bound.
\begin{lemma}\label{SSB}
  Whenever $\Delta_F\not=\square$ we have
  \[\mathfrak{S}(F)\ll_{\ve} \Dbad^{1/4+\ve}L(1,\chi_F)\ll_{\ve}
  \Dbad^{1/4}\|F\|^{\ve},\]
  where $\chi_F$ is the quadratic character defined by taking
  $\chi_F(2)=0$ and
\[\chi_F(p)=\left(\frac{\Delta_F}{p}\right)\]
for odd primes $p$.  
  \end{lemma}
\begin{proof}
  When $p\nmid 2\Delta_F$ we have
  \[S_{p^r}(\0)=\left(\frac{\Delta_F}{p^r}\right) p^{2r} \phi(p^r),\]
  by Lemma \ref{lem:good_primes}, whence
\[\prod_{p\nmid 2\Delta_F}\sum_{r=0}^{\infty}p^{-4r}S_{p^r}(\0)\ll
L(1,\chi_F).\]
The conductor of $\chi_F$ is $O(\|F\|^4)$, whence $L(1,\chi_F)\ll\log
(2+\|F\|)\ll_{\ve}\|F\|^{\ve}$. 

The factor corresponding to primes for which $p\mid \Delta_F$
and $p\nmid  2\Dbad$ is just 1, by Lemma \ref{lem:zero-c}.  It
therefore remains to consider primes $p\mid 2\Dbad$. Suppose that
 $p^{f_j}\|2A_j$ with $f_1\le f_2\le f_3\le f_4$, say. Then
Lemma \ref{lem:Sq-basic} yields
\[\sum_{r=0}^{\infty}p^{-4r}S_{p^r}(\0)\ll\sum_{r=0}^{\infty}p^{-r}
\prod_{1\le j\le 4}\min(p^{r/2},p^{f_j/2}).\]
If we bound the minimum by $p^{r/2}$ for $j=4$ we see that
\[\sum_{r=0}^{\infty}p^{-4r}S_{p^r}(\0)\ll\sum_{r=0}^{\infty}p^{-r/2}
\min(p^{3r/2},p^{(f_1+f_2+f_3)/2})\ll p^{(f_1+f_2+f_3)/3}.\]
If $p^e\|\Dbad$ this is $O(p^{e/4})$, so that primes which divide
$2\Dbad$ provide a total $O_{\ve}(\Dbad^{1/4+\ve})$ for any $\ve>0$.
The lemma then follows.
\end{proof}

\subsection{The main term}

We now collect our estimates together in order to complete the proof
of Theorem \ref{newth}. 
Let $\ve>0$ and let $C=\|F\|^{1/2}Q^\ve$.
Returning to \eqref{asym1}, it follows from part (i) of Lemma
\ref{lem:I-estimate} 
that 
\[N_w(P) = M(P)+E(P)+O(1),\]
where
\begin{align*}
M(P)
&=
\frac{1 }{Q^{2}} \sum_{q=1}^\infty q^{-4}S_q(\0)I_q(\0)
\quad \text{and} \quad
E(P)=
\frac{1 }{Q^{2}}\sum_{\substack{\c\in \ZZ^4\\ 0<|\c|\ll C}}
\sum_{q=1}^\infty q^{-4}S_q(\c)I_q(\c).
\end{align*}
In both $M(P)$ and $E(P)$ we recall that $I_q(\c)$ vanishes for
$q\gg Q$, by the properties of the function $h$.

To handle $M(P)$ our first task is to relate $I_q(\0)$ to the singular
integral, given by (\ref{eq:sigma-infty-omega}).
To begin with we note that 
\begin{align*}
I_q(\0)
&=P^4 \int_{\RR^4}w(\x)h\left(\frac{q}{Q},G(\x)\right)\d\x,
\end{align*}
where $G(\x)=\|F\|^{-1}F(\x)$. Since $w$ is compactly supported we
will have $|G(\x)|\le c_w$ whenever $w(\x)\not=0$, for some constant
$c_w$ depending only on $w$.  We choose a smooth weight function
$w_0:\RR\to\RR$ supported on $[-1-c_w,1+c_w]$ such that $w_0(t)=1$ for
$t\in[-c_w,c_w]$.  This choice can be made in an explicit way such that
$w_0$ depends only on $w$.  This allows us to write
\[I_q(\0)=
P^4 \int_{\RR^4}w(\x)w_0(G(\x))h\left(\frac{q}{Q},G(\x)\right)\d\x.\]
The function $f(t)=w_0(t)h(q/Q,t)$ is compactly supported with a
continuous second derivative.  Recall the definition \eqref{eq:Ku} of
the function $K(u;\delta)$.
The above condition is enough to ensure that
\[\int_{-\infty}^\infty f(t)K(t-\tau;\delta)\d t\to
f(\tau)\;\;\;\mbox{as } \delta\downarrow 0,\]
uniformly in $\tau$.  As a result one sees that
\[I_q(\0)=P^4\lim_{\delta\downarrow 0}\int_{\RR^4}\int_{-\infty}^{\infty}
w(\x)w_0(t)h\left(\frac{q}{Q},t\right)K(t-G(\x);\delta)\d t\d\x.\]

Using the equation (\ref{KFT}) we are now led
to the expression
\[I_q(\0)=P^4\lim_{\delta\downarrow 0}\int_{-\infty}^{\infty}
\left(\frac{\sin(\pi\delta\theta)}{\pi\delta\theta}\right)^2
J(\theta;w)L(\theta)\d\theta,\]
with $J(\theta;w)$ given by (\ref{J1d}) and
\[L(\theta)=\int_{-\infty}^{\infty}w_0(t)h\left(\frac{q}{Q},t\right)
e(\theta t)\d t.\]
The following result is concerned with estimating this integral.

\begin{lemma}\label{J2}
  Assume that $q\le Q$.  Then 
  \[L(\theta)=1+O_N((q/Q)^N)+O_N((q/Q)^N|\theta|^N)\]
  for any integer $N \ge 1$.
\end{lemma}

We will prove this at the end of this  section, but first we use it to
complete our treatment of $I_q(\0)$. Combining Lemma \ref{J1L} with
the trivial bound $J(\theta;w)\ll 1$ we have
$J(\theta;w)\ll_N(1+|\theta|)^{-2N}\|F\|^{4N\eta}$. 
The error terms in
Lemma \ref{J2} therefore contribute a total
\[\ll_N P^4\int_{-\infty}^{\infty}
\frac{(q/Q)^N\{1+|\theta|\}^N}{\{1+|\theta|\}^{2N}}\|F\|^{4N\eta}\d\theta
\ll_N P^4(q/Q)^N\|F\|^{4N\eta},\]
for any $N\ge 2$.  Moreover 
\[\int_{-\infty}^{\infty}
\left(\frac{\sin(\pi\delta\theta)}{\pi\delta\theta}\right)^2
J(\theta;w)\d\theta\to\int_{-\infty}^{\infty}J(\theta;w)\d\theta\]
as $\delta\downarrow 0$.  On replacing $\theta$ by $\|F\|\theta$ we
then see that
\[I_q(\0)=P^4\left\{\|F\|\sigma_\infty(w;F)+O_N((q/Q)^N\|F\|^{4N\eta})\right\}\]
for any $q\le Q$ and any $N\ge 2$. In particular, if 
$q\le Q\|F\|^{-5\eta}$ then, by taking $N$ suitably large, the error
term can be made smaller than any
given negative power of $P$, by virtue of 
our assumption that $\|F\|\ge P^{\eta}$.

We will need the following upper bound for the singular integral, in
which we do not make
either of the assumptions (\ref{eq:assume1-w}) or (\ref{eq:assume1-x}). 

\begin{lemma}\label{SIB}
Suppose that $\Delta_F\not=0$. Suppose either that $w$ is a smooth 
weight function of compact support, or that $w$ is the characteristic
function of $[-\kappa,\kappa]^4$ for some $\kappa>0$.  Then 
$\sigma_\infty(w;F)\ll_w |\Delta_F|^{-1/4}.$
\end{lemma}

We will prove this in the next section.  Taking the lemma as
proved, we see that (\ref{eq:assume1-x}) implies that
\beql{sib}
\sigma_\infty(w;F)\ll \|F\|^{-1+\eta}.
\eeq
We now have 
$$
M(P)=\frac{\|F\| \sigma_\infty(w;F) P^4 }{Q^{2}} 
\sum_{q\leq Q\|F\|^{-5\eta}}
q^{-4}S_q(\0)+
O\left(\left|T\left(Q\|F\|^{-5\eta}\right)\right|+1\right),
$$
where
\[T(M)=\frac{1}{Q^2}\sum_{q>M} q^{-4}S_q(\0)I_q(\0).\]
Summation by parts yields
\[T(M)=-\frac{I_M(\0)}{Q^2 M}\Sigma(M;\0)-
  \frac{1}{Q^2}\int_M^{\infty}\Sigma(x;\0)
  \frac{\partial}{\partial x} \frac{I_x(\0)}{x} \d x, \]
where $\Sigma(x;\0)$ is given by \eqref{eq:Sigma}.
However $I_x(\0)$ vanishes identically when $x\gg Q$, whence 
it follows from  Lemmas  \ref{lem:I0} and  \ref{lem:average2} that 
\[T(M)\ll \frac{\|F\|^\eta P^4}{Q^2 M}\sup_{M\leq x\ll Q} |\Sigma(x;\0)|
\ll_\ve \frac{\|F\|^\eta P^4}{Q^2 M}|\Delta_F|^{3/16+\ve} 
\Dbad^{3/8}Q^{1/2+\ve}.\]
Taking $M=Q\|F\|^{-5\eta}$ 
with $Q=P\|F\|^{1/2}$, and using the bounds
$|\Delta_F|\le  \|F\|^4$ and $\Dbad\le\|F\|^{\eta}$, this yields
\[T(Q\|F\|^{-5\eta})
\ll_\ve \|F\|^{-1/2+51\eta/8+\ve}P^{3/2+\ve}.\]
Our assumption that $\|F\|\ge P^{\eta}$ enables us to replace
$P^{\ve}$ by $\|F\|^{\ve/\eta}$, whence, on choosing $\ve$ suitably
small, we obtain 
the bound
$$
M(P)=\frac{\|F\| \sigma_\infty(w;F) P^4 }{Q^{2}} 
\sum_{q\leq Q\|F\|^{-5\eta}}
q^{-4}S_q(\0)+O\left(
\|F\|^{-1/2+7\eta}P^{3/2}\right).
$$
A similar analysis shows that
\[\frac{\|F\| P^4 }{Q^{2}} 
\sum_{q\leq Q\|F\|^{-5\eta}}
q^{-4}S_q(\0)=\frac{\|F\| P^4 }{Q^{2}} \sum_{q=1}^{\infty}
q^{-4}S_q(\0)+
O(\|F\|^{1/2+6\eta}P^{3/2}).\]  
The infinite sum is just $\mathfrak{S}(F)$.  Hence, using (\ref{sib}),
we find that
\[M(P)=\sigma_\infty(w;F) \mathfrak{S}(F)P^2+
O\left(\|F\|^{-1/2+7\eta} P^{3/2}\right).\]
This is satisfactory for the statement of Theorem \ref{newth}.

\begin{proof}[Proof of Lemma \ref{J2}]
When $|\theta|\le 1$
the result is an immediate application of Heath-Brown \cite[Lemma~9]{HB'}. 
Moreover, taking $N=2$ in (\ref{eq:h-function}) yields
$h(x,t)\ll x+\min(x^{-1},xt^{-2})$, 
whence one trivially has $L(\theta)\ll 1$ for
$q\le Q$.  We may therefore assume that $qQ^{-1}|\theta|\le 1\le|\theta|$.
In this case we must modify the
proof of \cite[Lemma~9]{HB'}. It will be convenient to write $x=q/Q$
and $X=\sqrt{x/|\theta|}$.
Since $w_0$ has compact support we may
suppose that $t\ll_w 1$.  Thus (\ref{eq:h-function}) implies that 
$h(x,t)\ll_N x^{N-1}|t|^{-N}$ for any $N\ge 1$.  It then follows that
the range $|t|\ge X$ contributes
$O_N(x^{N-1}X^{1-N})=O_N((x|\theta|)^{(N-1)/2})$ to $L(\theta)$.
This is satisfactory for the error terms of Lemma~\ref{J2}, on
redefining $N$.

When $|t|\le X$ we use Taylor's theorem to approximate
$w_0(t)e(\theta t)$ by a polynomial of degree $M$, say, together with
an error $O_M(X^{M+1}|\theta|^{M+1})$. Since we have
$h(x,t)\ll x+\min(x^{-1},xt^{-2})$, as
noted above, this error term contributes
$O_M((x|\theta|)^{(M+1)/2})$ to $L(\theta)$, which again is
satisfactory if $M$ is large enough. The polynomial produced by
Taylor's theorem has terms $c_mt^m$ with $c_m\ll_M 
|\theta|^m$.  When
$1\le m\le M$ we apply \cite[Lemma~8]{HB'}, producing an overall bound
$O_M((X|\theta|)^m(x/X)^M)$ for each value of $m$.  One should note
here firstly that the required condition $x\ll\min(1,X)$ holds, by
virtue of the condition $x|\theta|\le 1$, and secondly that
$Xx^{M-1}\le x^MX^{-M}$ for $M\ge 1$, since $x\le 1\le|\theta|$.
On considering the possible values for $m\in \{1,\dots,M\}$  
we then see that each 
monomial $c_mt^m$ contributes 
$O_M((x|\theta|)^M)+O_M((x|\theta|)^{(M+1)/2})$, 
which is satisfactory when
$M$ is taken large enough.  Finally, the constant term $c_0$ is
handled analogously using \cite[Lemma~6]{HB'}, producing the same
error term together with a main term $c_0$. However the function $w_0$
takes the value 1 at the origin, and the lemma follows.
\end{proof}

\subsection{The singular integral} 
We begin
  by proving Lemma \ref{SIB}. The function
\[J(\theta;w)=\int_{\RR^4}w(\x)e(-\theta G(\x))\d\x\]
is well-defined and continuous. We claim that
$J(\theta;w)\in L^1(\RR)$ for $w$ as in the lemma. If $w$ is smooth
and supported in $[-\kappa,\kappa]^4$ then
\[J(\theta;w)\ll_w\int_{[-\kappa,\kappa]^4}e(-\theta G(\x)+
\x. \y)\d\x\] 
for some $\y\in\RR^4$, by Lemma 3.2 of Heath-Brown and Pierce
\cite{hb-pierce}.  The integral on the right factors into four
1-dimensional integrals, with
\[\int_{-\kappa}^{\kappa}e(-\theta A_j\|F\|^{-1}x^2+xy_j) 
\d x\ll_{\kappa}
\min\left\{1\,,\,\sqrt{\frac{\|F\|}{|\theta A_j|}}\right\}.\]
It then follows that
\[J(\theta;w)\ll\prod_{j=1}^4
\min\left\{1\,,\,\sqrt{\frac{\|F\|}{|\theta A_j|}}\right\}\ll
\min\left\{1\,,\,\frac{\|F\|^2}{\theta^2|\Delta_F|^{1/2}}\right\},\]
whence
\[\int_{-\infty}^{\infty}J(\theta;w)\d\theta  \ll
\int_{-\infty}^{\infty}
\min\left\{1\,,\,\frac{\|F\|^2}{\theta^2|\Delta_F|^{1/2}}\right\}
\d\theta  \ll \frac{\|F\|}{|\Delta_F|^{1/4}},\]
as required.
\bigskip

We now use the machinery developed above to see how to compare
$\sigma_\infty(w;F)$ for different weights.

\begin{lemma}\label{csw}
Let  $w_0(\x)$ be the
  characteristic function of the region $[-1,1]^4$.  
 Suppose that $w_1(\x)$ (respectively, $w_2(\x)$) is supported in the region
  $\eta\le|\x|\le 1$ (respectively, $\eta\le|\x|\le 1+\eta$)
    and takes values in $[0,1]$ there.  
Suppose further that $w_1(\x)=1$ whenever
  $2\eta\le|\x|\le 1-\eta$ 
(respectively,  $w_2(\x)=1$ whenever
  $2\eta\le|\x|\le 1$). Then
    \[\sigma_\infty(w_i;F)=\sigma_\infty(w_0;F)+O(\eta^{1/2}|\Delta_F|^{-1/4}),\]
    for $i=1,2$,    the implied constant being absolute.
\end{lemma}

\begin{proof}
  We confine ourselves to the result for $w_1$, the function $w_2$
  being treated analogously.  
We first calculate that 
\begin{align*}
  \int_{-\infty}^{\infty}J(\theta;w)\d\theta
 &=\lim_{\delta\downarrow 0}\int_{-\infty}^{\infty}
\left(\frac{\sin(\pi\delta\theta)}{\pi\delta\theta}\right)^2
J(\theta;w)\d\theta\\
&=\lim_{\delta\downarrow 0}\int_{-\infty}^{\infty}\int_{\RR^4}w(\x)e(-\theta G(\x))
\left(\frac{\sin(\pi\delta\theta)}{\pi\delta\theta}\right)^2
\d\x\d\theta.
\end{align*}
The conditions for Fubini's Theorem are satisfied, allowing us to
switch the two integrations.  The relation (\ref{KFT}) then shows us
that
\beql{fund}
\int_{-\infty}^{\infty}J(\theta;w)\d\theta=
\lim_{\delta\downarrow 0}\int_{\RR^4}w(\x)K(-G(\x);\delta)\d\x.
\eeq
  It follows that 
  \[\|F\|\{\sigma_\infty(w_1;F)-\sigma_\infty(w_0;F)\}=
\lim_{\delta\downarrow 0}\int_{\RR^4}\{w_1(\x)-w_0(\x)\}K(-G(\x);\delta)\d\x.\]
We have $|w_1(\x)-w_0(\x)|\le 1$ for all $\x$, the difference being
non-zero only if either $|\x|\le 2\eta$ or there is some index $i$ for
which $1-\eta\le |x_i|\le 1$. Hence
\[|w_1(\x)-w_0(\x)|\le \sum_{n=0}^{8}f_n(\x),\] 
where each $f_n(\x)$ is the
characteristic function of a certain box $I_1\times\ldots \times I_4$.  For
$n=0$ this is just $[-2\eta,2\eta]^4$, but otherwise 3 of the
intervals $I_j$ have length 2, and the fourth has length $\eta$.
Thus
\[\|F\|.\left|\sigma_\infty(w_1;F)-\sigma_\infty(w_0;F)\right|\le \sum_{n=0}^{8}
\lim_{\delta\downarrow 0}\int_{\RR^4}f_n(\x)K(-G(\x);\delta)\d\x.\]
The expression on the right may be evaluated via a further application
of (\ref{fund}). As in
the proof of Lemma \ref{SIB}, we have
\[\int_{I_j}e(\theta A_j\|F\|^{-1}x^2)\d x\ll
\min\left\{\meas(I_j)\,,\,\sqrt{\frac{\|F\|}{|\theta A_j|}}\right\},\]
leading to a bound
\begin{align*}
\|F\|\{\sigma_\infty(w_1;F)-\sigma_\infty(w_0;F)\}
&\ll  \int_{-\infty}^{\infty}
  \min\left\{\eta\,,\, \frac{\|F\|^2}{\theta^2
    |\Delta_F|^{1/2}}\right\} \d \theta\\
  &\ll \eta^{1/2} \frac{\|F\|}{ |\Delta_F|^{1/4}}
  \end{align*}
  and the lemma follows.
  \end{proof}

It is convenient at this point to record some facts about
$\tau_{\infty}$, as given in Theorem \ref{t:main}.
\begin{lemma}\label{tinf}
  Let $\rho_\infty(\y)$ be given by \eqref{eq:def-rho}, and define
  \beql{siD}
  \sigma_\infty(\x)=\int_{-\infty}^\infty \int_{[-1,1]^4}
  e(-\theta F(\x;\y)) \d \y \d \theta.
  \eeq
  Then 
  \beql{tau2}
  \tau_\infty=\int_{[-1,1]^4}\rho_\infty(\y)\d\y
  =\int_{[-1,1]^4}\sigma_\infty(\x)\d\x,
  \eeq
Moreover, these last two integrals are absolutely convergent. 
\end{lemma}
\begin{proof}
In order to verify (\ref{tau2}) it is enough to confirm that the orders
of integration may be suitably changed, and this will be permissible provided
both
\[\int_{[-1,1]^4}e(-\theta F(\x;\y))\d\x\;\;\;\mbox{and}
\;\;\;\int_{[-1,1]^4}e(-\theta F(\x;\y))\d\y\]
are in $L^1(\RR\times [-1,1]^4)$.  However 
these are of order 
\[\prod_{i=1}^4\min\{1\,,\,|\theta|^{-1}|y_i|^{-2}\}\;\;\;\mbox{and} 
\;\;\;\prod_{i=1}^4\min\{1\,,\,|\theta|^{-1/2}|x_i|^{-1/2}\}\] 
respectively, and so are both integrable. 
\end{proof}

\subsection{Analysis of $E(P)$}\label{ssEB}

The key observation in handling $E(P)$ is that the hypothesis
$\|F\|\ge P^{\eta}$ in Theorem \ref{newth}, along with the assumption
(\ref{eq:assume2-x}),  allow us to restrict to $\c\in \ZZ^4$ for which
$F^*(\c)\neq 0$.
Indeed, suppose for a contradiction that $\c\in \ZZ^4$ is such that
$0<|\c|\ll C$ and $F^*(\c)=0$, where $F^*(\c)$ is given by \eqref{eq:F*}. 
Let us assume, for example, that $c_1\not=0$. Now, if we write
\[A_1^\flat=\prod_{\substack{p\| A_1\\ p\nmid A_2A_3A_4 }}p,\]
then the equation $F^*(\c)=0$ implies that
$A_1^\flat \mid c_1$. Since $1\le|c_1|\ll C$ we would then deduce that
$A_1^\flat\ll C$.  However \eqref{eq:assume1-x} and
\eqref{eq:assume2-x} yield
\[A_1^\flat\geq \frac{|A_1|}{\Dbad} \geq \|F\|^{1-2\eta}.\] 
Since $C=\|F\|^{1/2}Q^{\ve}$ with $Q=\|F\|^{1/2}P$, we therefore
obtain a contradiction if $\ve$ is small enough and $P$ is large
enough, since $P^{\eta}\le \|F\|$.

We now have
\[E(P)=\frac{1 }{Q^{2}}\sum_{\substack{\c\in \ZZ^4\\ 
|\c|\ll C\\ F^*(\c)\neq 0}} \sum_{q=1}^\infty q^{-4}S_q(\c)I_q(\c).\]
The analysis of quaternary quadratic forms, in Heath-Brown
\cite[\S 12]{HB'} for example, normally requires one to obtain
some cancellation from the summation over $q$, 
but this is no 
longer necessary  
because we have 
been able to remove vectors $\c$ for which $F^*(\c)=0$.

Noting that  $I_q(\c)$ is only supported on $q\ll Q$, we 
deduce from part (ii) of Lemma \ref{lem:I-estimate}
that 
\[\sum_{q=1}^\infty q^{-4}S_q(\c)I_q(\c) \ll 
\frac{\|F\|^{2}P^{3}}{|\Delta_F|^{1/2}|\c|}
\sum_{q\ll Q} q^{-3}|S_q(\c)|.\]
Lemma \ref{lem:average1} now implies that
$$
E(P)\ll 
\frac{\|F\|^{2+\ve} P^{3+\ve}}{|\Delta_F|^{1/2} Q^2}
\sum_{\substack{\c\in \ZZ^4\\ 0<|\c|\ll C}} \frac{
\prod_{1\leq i\leq 4} (A_i,c_i)^{1/2}}{|\c|} .
$$
It remains to estimate the $\c$-sum, which we temporarily denote by
$K$. We plainly have  
\[K\leq \sum_{d_i\mid A_i} \sqrt{d_1d_2d_3d_4} 
\sum_{\substack{\c\in \ZZ^4\setminus \{\0\}\\ c_i\ll C/d_i } }
\frac{1}{\max_{1\leq i\leq 4} |d_ic_i|}.\]
The inner sum is 
\[\ll \left(\frac{C}{d_i}+1\right)\left(\frac{C}{d_j}+1\right)
\left(\frac{C}{d_k}+1\right)  \frac{\log C}{d_l}\]
for some permutation $\{i,j,k,l\}$ of $\{1,2,3,4\}$.
Multiplying this by $\sqrt{d_1d_2d_3d_4}$ and recalling that
 $C=\|F\|^{1/2}Q^\ve$, this gives
\[K \ll \sum_{d_i\mid A_i} \frac{\log C}{\sqrt{d_l}}
 \left(C+\sqrt{d_i}\right)\left(C+\sqrt{d_j}\right)\left(C+\sqrt{d_k}\right) 
\ll \|F\|^{3/2+\ve}P^{\ve},\] 
on employing the trivial estimate for the  divisor function.

Absorbing $P^\ve$ into $\|F\|^\ve$, it now follows that 
$$
E(P)\ll 
\frac{\|F\|^{5/2+\ve} P}{|\Delta_F|^{1/2}}.
$$
Our hypotheses \eqref{eq:assume1-x} implies that 
$|\Delta_F|\geq \|F\|^{4-3\eta}$, from which one sees that our bound for 
$E(P)$ is satisfactory for the error term given in Theorem
\ref{newth}.

\section{Combining the various ingredients}\label{s:combine}

Theorems \ref{t:asymplatt} and \ref{newth} will be our main
tools for the proof of Theorem \ref{t:main}.
In this section we begin by 
adapting them so as to count only
primitive vectors.  We then
apply Theorem \ref{newth} to the quadrics in $\y$ given by
$F(\x;\y)=0$, and sum the resulting asymptotic formulae with respect
to $\x$.  In doing so we must allow for those vectors $\x$ excluded by
the conditions of the theorem.  The next stage is to remove the weight
$w$ occuring in Theorem \ref{newth}.  We then piece together our two main
estimates, and make suitable adjustments so as to cover all
primitive $\x,\y$ with $|\x|^3|\y|^2\le B$.  In \S \ref{s:reckon}
 we will show
that the main terms combine to give  $cB\log B$,  
as $B\to \infty$,  with the  
constant $c$ given by (\ref{PC}).

\subsection{Primitive solutions with small $\y$}

In analogy to (\ref{eq:deg1}) and (\ref{sm}) we define
\[M_5(\y;R)=\#\left\{ \x\in \ZZ^4: |\x|\leq R, ~
\Delta(\x)\not=\square,~ F(\x;\y)=0 \right\}\]
and
\[N_2(B;Y)=\sum_{\substack{\y\in \ZZp^4\\ Y<|\y|\leq 2Y}}
M_5\left(\y;(B/|\y|^2)^{1/3}\right),\]
so that
\beql{ee2}
N_2(B;Y)\ll B+B^{5/3}Y^{-8/3},
\eeq
by the first part of Lemma \ref{lem:stage1}.

We may also estimate $N_2(B;Y)$ by removing solutions with
$\Delta(\x)=\square$ from the counting function $N_1(B;Y)$ given by
Theorem \ref{t:asymplatt}.  According to Lemma~\ref{lem:xz} there are
$O_{\ve}(BY^{-1+\ve}+Y^4)$ solutions with $\Delta(\x)=0$. Moreover
there are $O(X^2)$ possible square values for $\Delta(\x)\not=0$ when
$|\x|\le X$, each such value corresponding to $O_{\ve}(X^{\ve})$
vectors $\x$.  Thus Lemma \ref{ann} shows that solutions with
$\Delta(\x)=\square\not=0$ contribute $O_{\ve}(B^{2/3+\ve}Y^{2/3})$ to
$N_1(B;Y)$.  Taking $\ve=\frac{2}{15}$ we deduce that 
$$
B^{2/3+\ve}Y^{2/3}=(B^{2/3}Y^{4/3})^{3/5}(BY^{-1/3})^{2/5}
\leq B^{2/3}Y^{4/3}+ BY^{-1/3}.
$$
It then follows from Theorem \ref{t:asymplatt} that
\beql{ee1}
N_2(B;Y)=B\sum_{\substack{\y\in \ZZp^4\\ Y<|\y|\leq 2Y}}
\frac{\rho_\infty(\y)}{|\y|^2}+O(B^{2/3}Y^{4/3})+O(BY^{-1/3})+O(Y^4)
\eeq
for $Y\ge \tfrac12$,  
where $\rho_\infty(\y)$ is given by \eqref{eq:def-rho}.

We now set
\[M_6(\y;R)=\#\left\{ \x\in \ZZp^4: |\x|\leq R,
~\Delta(\x)\not=\square, ~ F(\x;\y)=0
\right\}\]
and
\[N_3(B;Y)=\sum_{\substack{\y\in \ZZp^4\\ Y<|\y|\leq 2Y}}
M_6\left(\y;(B/|\y|^2)^{1/3}\right),\]
whence
\[M_6(\y;R)=\sum_{d\le R}\mu(d)M_5(\y;R/d).\]
Our goal now is the following estimate.

\begin{lemma}\label{N3A}
    If $\tfrac12\le Y\le B^{1/4}$ we have  
\[N_3(B;Y)=\frac{B}{\zeta(3)}\sum_{\substack{\y\in \ZZp^4\\ Y<|\y|\leq 2Y}}
\frac{\rho_\infty(\y)}{|\y|^2}+O(B^{2/3}Y^{4/3})+O(BY^{-1/3}).\]
\end{lemma}
  \begin{proof}
We start from the relation
\[N_3(B;Y)=\sum_{d\le (B/Y^2)^{1/3}}\mu(d)N_2(B/d^3;Y).\]
To estimate this sum we choose a parameter $D$ in the range 
$1\le D\le B^{1/3}Y^{-2/3}$ and use (\ref{ee1}) for $d\le D$ and 
(\ref{ee2}) for $d>D$. This yields
\begin{align*}
  N_3(B;Y)&=B\left(\sum_{d\le D}\frac{\mu(d)}{d^3}\right)
\sum_{\substack{\y\in \ZZp^4\\ Y<|\y|\leq 2Y}}
\frac{\rho_\infty(\y)}{|\y|^2}+O(B^{2/3}Y^{4/3})+O(BY^{-1/3})\\
& \hspace{1cm} +O(Y^4D)+O(BD^{-2})+O(B^{5/3}D^{-4}Y^{-8/3}).
\end{align*}
It follows from \eqref{eq:2constants} that 
\[
\sum_{\substack{\y\in \ZZp^4\\ Y<|\y|\leq 2Y}}\frac{\rho_\infty(\y)}{|\y|^2}
=\sum_{\substack{\y\in \ZZp^4\\ Y<|\y|\leq 2Y}}\frac{V(\y)}{|\y|^2\mathsf{d}(\y)}
\ll 1,\]
so  that the leading term is
\[\frac{B}{\zeta(3)}\sum_{\substack{\y\in \ZZp^4\\ Y<|\y|\leq 2Y}}
\frac{\rho_\infty(\y)}{|\y|^2}+O(BD^{-2}).\]
Thus if  we choose $D=B^{1/3}Y^{-4/3}$ we obtain 
\[N_3(B;Y)=\frac{B}{\zeta(3)}\sum_{\substack{\y\in \ZZp^4\\ Y<|\y|\leq 2Y}} 
\frac{\rho_\infty(\y)}{|\y|^2}+O(B^{2/3}Y^{4/3})+O(BY^{-1/3})+O(B^{1/3}Y^{8/3}).\] 
Since $Y\le B^{1/4}$ the final error term is bounded by the first,
as required.
\end{proof}

  \subsection{Primitive solutions for typical small $\x$}
  
We next perform a similar computation for solutions in which $\x$ is
small and $\y$ is large, counted via the fibration into quadrics, 
using Theorem \ref{newth}. We write
\[M_7(\x;P,w)=\sum_{\substack{\y\in \ZZ^4\\ F(\x;\y)=0}} w(P^{-1}\y),\]
where  $w(\y)$ is an infinitely differentiable  weight function 
of compact support that vanishes 
for $|\y|\le\eta$.  
Then  Theorem \ref{newth} shows that
\[M_7(\x;P,w)=\sigma_\infty(\x;w)\mathfrak{S}(\x)P^{2}+
  O_{w,\eta}\left(|\x|^{-1/2+7\eta}P^{3/2}+|\x|^{1/2+2\eta}P \right)\]
  when $|\x|\ge P^{\eta}$, provided that
  \begin{equation}\label{eq:assume3-x}
|\x|^{1-\eta}\leq |x_i|\big(\leq |\x|\big), \quad \text{for  $1\leq i\leq 4$},
\end{equation}
and  that
\begin{equation}\label{eq:assume4-x}
\Dbad(\x) \leq |\x|^{\eta}.
\end{equation}
The singular integral and series are given by
\beql{siD2}
\sigma_\infty(\x;w)=\int_{-\infty}^{\infty}\int_{\RR^4}w(\y) 
e(-\theta F(\x;\y))\d\y\d\theta
\eeq
and
\[\mathfrak{S}(\x)=\prod_p \lim_{r\to\infty} p^{-3r}\#\{\y\in
(\ZZ/p^r\ZZ)^4: F(\x;\y)\equiv 0\bmod{p^r}\}. \]
We now write
\[M_8(\x;P,w)=\sum_{\substack{\y\in \ZZp^4\\ F(\x;\y)=0}}
w(P^{-1}\y)\]
and proceed to derive the following estimate.

\begin{lemma}\label{M8A}
Suppose that $\x$ satisfies the conditions \eqref{eq:assume3-x} and
\eqref{eq:assume4-x}, and that $P^{\eta}\le |\x|\le P^{2/3}$. We then have
\beql{M8B}
M_8(\x;P,w)=\frac{\sigma_\infty(\x;w)\mathfrak{S}(\x)}{\zeta(2)}P^2
+O_{w,\eta}(|\x|^{-1/2}P^{5/3+5\eta}).
+O_{w,\eta}(|\x|^{-1/2}P^{5/3+4\eta}). 
 \eeq
\end{lemma}
\begin{proof}
Our starting point is the relation
\begin{align*}
  M_8(\x;P,w)=~&\sum_{d\ll P}\mu(d)M_7(\x;P/d,w)\\ 
  =~&  \sigma_\infty(\x;w)\mathfrak{S}(\x)P^{2}\{\zeta(2)^{-1}+O(P^{-1})\}\\
  &+
  O_{w,\eta}\left(|\x|^{-1/2+7\eta}P^{3/2}+|\x|^{1/2+2\eta}P^{1+\eta}\right).
  \end{align*}
If we assume that $|\x|\le P^{2/3}$ the final error term is
$O_{w,\eta}(|\x|^{-1/2}P^{5/3+5\eta})$. 
We also observe that
$\sigma_{\infty}(\x;w)\ll |\x|^{-1+\eta}$, as in (\ref{sib}), and that
$\mathfrak{S}(\x)\ll|\x|^{\eta} $ by Lemma \ref{SSB} and
(\ref{eq:assume4-x}). The estimate (\ref{M8B}) then follows.
\end{proof}

We are now ready to consider the average
\[N_4(B;X,w)=\sum_{\substack{\x\in \ZZp^4,\,\Delta(\x)\not=\square\\ X<|\x|\leq 2X}}
M_8\left(\x;(B/|\x|^3)^{1/2},w\right),\]
for which we have the following estimate.

\begin{lemma}\label{N4A}
  Let
  \beql{XB}
B^{2\eta}\le X\le B^{1/6-4\eta}.
B^{2\eta}\le X\le B^{1/6-3\eta}. 
  \eeq
  Then if $\eta$ is small enough we will have
\[N_4(B;X,w)=\frac{B}{\zeta(2)}\sum_{\substack{\x\in
    \ZZp^4,\,\Delta(\x)\not=\square\\ X<|\x|\leq 2X}}
\frac{\sigma_\infty(\x;w)\mathfrak{S}(\x)}{|\x|^3}+O_{w,\eta}(B^{1-\eta^2/20}).\]
\end{lemma}
\begin{proof}
  We would like to apply Lemma \ref{M8A} for those vectors $\x$ which
  satisfy the conditions (\ref{eq:assume3-x}) and
  (\ref{eq:assume4-x}), and for which
  $(B/|\x|^3)^{\eta/2}\le|\x|\le(B/|\x|^3)^{1/3}$. This final
  constraint holds if $X<|\x|\le 2X$ with $X$ satisfying (\ref{XB}).
  Moreover, the error term contributes
a total $O_{w,\eta}(B^{5/ 6+3\eta}X)$, 
which is satisfactory for $X$ 
in the range \eqref{XB}.  
  Thus to complete our treatment
  of $N_4(B;X,w)$ we must consider vectors $\x$ for which either
  (\ref{eq:assume3-x}) or (\ref{eq:assume4-x}) fails.

  We  begin by considering the number of
  solutions $(\x,\y)$ for such $\x$. 
  By the third part of Lemma~\ref{lem:stage1} the number of
  solutions $(\x,\y)$ for 
  which (\ref{eq:assume3-x}) fails will be
  \[\ll_{\ve} B^\ve X^{-3\eta/4}\{B+B^{1/2}X^{5/2}\} 
 \ll_{\ve} B^\ve X^{-3\eta/4}\{B+B^{11/12}\}  
   \ll_{\ve} B^{1+\ve} X^{-3\eta/4}.\]  
This is satisfactory when $B^{2\eta}\le X\le B^{1/6-4\eta}$, 
   provided that we take $\ve\le\eta^2$.  Similarly, by Lemma 
  \ref{lem:bad_coeffs}, the number of
  solutions $(\x,\y)$ for which 
    (\ref{eq:assume4-x}) fails will be
  \[\ll_\ve B^\ve\left\{BX^{-\eta/24}+B^{1/2}X^{5/2} \right\}\]
  for any fixed $\ve>0$. 
    As before, under the assumption (\ref{XB}) this becomes 
  $O(B^{1-\eta^2/20})$ if we choose $\ve$ small enough.  Thus
  vectors $\x$ which fail to satisfy either (\ref{eq:assume3-x}) or
  (\ref{eq:assume4-x}) will make a suitably small contribution to
  $N_4(B;X,w)$. 

  To complete the proof of Lemma \ref{N4A} it remains to prove that 
  \beql{AA}
  \sum_{\substack{\x\in \ZZp^4,\,\Delta(\x)\not=\square\\ X<|\x|\leq 2X \\
  \text{\eqref{eq:assume3-x} or
  \eqref{eq:assume4-x} fails }}}
  \frac{\sigma_\infty(\x;w)\mathfrak{S}(\x)}{|\x|^3}\ll X^{-\eta/5}, 
  \eeq
  since for $X$  in the range
(\ref{XB})  the right hand side will then be
$O(B^{-2\eta^2/5})$, which is satisfactory for Lemma \ref{N4A}.
According to Lemmas \ref{SSB} and \ref{SIB} we have
\[
\frac{\sigma_\infty(\x;w)\mathfrak{S}(\x)}{|\x|^3}\ll_{\ve}
X^{-3+\ve}\frac{\Dbad(\x)^{1/4}}{|x_1x_2x_3x_4|^{1/4}}.\]
Let $s$ run over square-full positive integers, and write
$s=\Dbad(\x)$ and $n=|x_1x_2x_3x_4|$. Then vectors for which
(\ref{eq:assume3-x}) fails will contribute
\begin{align*}
  &\ll_{\ve} X^{-3+\ve}\sum_s s^{1/4}
  \sum_{\substack{n\le 16X^{4-\eta}\\ s\mid n}}\tau_4(n)n^{-1/4}\\
  &\ll_{\ve} X^{-3+2\ve}\sum_s \sum_{m\le 16X^{4-\eta}/s}m^{-1/4}\\
  &\ll_{\ve} X^{-3+2\ve}\sum_s X^{3-3\eta/4}s^{-3/4}.
\end{align*}
However, if $s$ runs over square-full integers the infinite sum
$\sum s^{-3/4}$ converges, so that the above will be $O(X^{-\eta/2})$
if we choose $\ve$ small enough, which is satisfactory for
 (\ref{AA}). Similarly, vectors for which
(\ref{eq:assume4-x}) fails will contribute
\begin{align*}
  &\ll_{\ve} X^{-3+\ve}\sum_{s\ge X^{\eta}} s^{1/4}
  \sum_{\substack{n\le 16X^4\\ s\mid n}}\tau_4(n)n^{-1/4}\\
  &\ll_{\ve} X^{-3+2\ve}\sum_{s\ge X^{\eta}} \sum_{m\le 16X^4/s}m^{-1/4}\\
  &\ll_{\ve} X^{-3+2\ve}\sum_{s\ge X^{\eta}} X^3s^{-3/4}.
\end{align*}
Since
\[\sum_{s\ge S}s^{-3/4}\ll S^{-1/4}\]
for any $S\ge 1$, the above will be 
$O(X^{-\eta/5})$ for small
enough $\ve$, which again produces a satisfactory contribution to
(\ref{AA}).  This completes the proof of the lemma.
\end{proof}

\subsection{Removing the weights}

The counting function $N_4(B;X,w)$ involves a weight function $w$, and
our next task is to remove it so as to produce 
\[N_5(B;X)=\sum_{\substack{\x\in \ZZp^4,\,\Delta(\x)\not=\square\\ X<|\x|\leq 2X}}
\sum_{\substack{\y\in\ZZp^4\\ F(\x;\y)=0}}w_0((B/|\x|^3)^{-1/2}\y),\] 
where $w_0$ is the characteristic function of $[-1,1]^4$,
as in Lemma \ref{csw}. The result is described in the 
following estimate.
\begin{lemma}\label{N56}
  If $X$ is in the range \eqref{XB} we have
  \[N_5(B;X)=\frac{B}{\zeta(2)}\sum_{\substack{\x\in
    \ZZp^4,\,\Delta(\x)\not=\square\\ X<|\x|\leq 2X}}
\frac{\sigma_\infty(\x)\mathfrak{S}(\x)}{|\x|^3}+  
O(\eta^{1/2}B)+O_{\eta}(B^{1-\eta^2/20}).\]
\end{lemma}

We remark that $\sigma_{\infty}(\x)=\sigma_{\infty}(\x;w_0)$ in view
of (\ref{siD}) and (\ref{siD2}). 

\begin{proof} 
Given $\eta\in(0,\tfrac{1}{100})$ 
we can construct specific weights
$w_1,w_2$ depending on $\eta$ alone, and satisfying the conditions of
Lemma \ref{csw}. Thus for all
$\u$  we have  $0\le w_1(\u), w_2(\u)\le 1$. Both functions vanish
when $|\u|\le \eta$.  The weight $w_1$
takes the value 1 for $2\eta\le|\u|\le 1-\eta$ and vanishes for
$|\u|\ge 1$;  the weight $w_2$ takes the value 1 for
$2\eta\le|\u|\le 1$ and vanishes for $|\u|\ge 1+\eta$.
In
particular $0\le w_1(\u)\le w_0(\u)$ for all $\u$, so that
$N_4(B;X,w_1)\le N_5(B;X)$.  The condition that
$|(B/|\x|^3)^{-1/2}\y|\le 2\eta$ is equivalent to the condition 
$|(B'/|\x|^3)^{-1/2}\y|\le 1$ with $B'=4\eta^2B$, whence 
\[N_5(B;X)-N_5(4\eta^2B;X)\le N_4(B;X,w_2).\]
Since the first part of Lemma \ref{lem:stage1} shows that
\[N_5(4\eta^2B;X)\ll \eta^{2/3}B\]
for $X\le B^{1/6}$, we see that it will suffice to show that
\[N_4(B;X,w_i)=\frac{B}{\zeta(2)}\sum_{\substack{\x\in
    \ZZp^4,\,\Delta(\x)\not=\square\\ X<|\x|\leq 2X}}
\frac{\sigma_\infty(\x)\mathfrak{S}(\x)}{|\x|^3}+
O(\eta^{1/2}B)+O_{\eta}(B^{1-\eta^2/20}).\]
for $i=1,2$.  However according to Lemma \ref{N4A} we have
\[N_4(B;X,w_i)=\frac{B}{\zeta(2)}\sum_{\substack{\x\in
    \ZZp^4,\,\Delta(\x)\not=\square\\ X<|\x|\leq 2X}}
\frac{\sigma_\infty(\x;w_i)\mathfrak{S}(\x)}{|\x|^3}+O_{\eta}(B^{1-\eta^2/20})\]
for $i=1,2$.  It needs to be stressed at this point that the implied
constant for the error term depends only on $\eta$, since the two
weight functions are completely fixed once $\eta$ is chosen. Moreover
our two weight functions do indeed vanish on a neighbourhood of the
origin as was required at the outset in \S\ref{CPOQ}.

We now use Lemma \ref{csw} to replace $\sigma_\infty(\x;w_i)$ by
$\sigma_\infty(\x)$, introducing an error $O(\eta^{1/2}BS(X))$ with
\beql{SD*}
S(X)=\sum_{\substack{\x\in\ZZp^4,\,\Delta(\x)\not=\square\\ X<|\x|\leq 2X}}
|x_1x_2x_3x_4|^{-1/4}\frac{\mathfrak{S}(\x)}{|\x|^3}.
\eeq
We therefore deduce that
\begin{equation}\label{eq:waypoint}
\begin{split}
N_5(B;X)&=\frac{B}{\zeta(2)}\sum_{\substack{\x\in
    \ZZp^4,\,\Delta(\x)\not=\square\\ X<|\x|\leq 2X}}
\frac{\sigma_\infty(\x)\mathfrak{S}(\x)}{|\x|^3}\\
&\hspace{1cm}+O(\eta^{1/2}B)+O(\eta^{1/2}BS(X))+O_{\eta}(B^{1-\eta^2/20}). 
\end{split}
\end{equation}
In order to estimate the sum $S(X)$ we apply Lemma \ref{SSB}
with $\ve=\frac{1}{20}$, which yields
$$
S(X)\ll X^{-3} 
\sum_{\substack{\x\in\ZZp^4,\,\Delta(\x)\not=\square\\ X<|\x|\leq 2X}}
|x_1x_2x_3x_4|^{-1/4} \Dbad(\x)^{3/10}L(1,\chi).
$$
We proceed by mimicking the proof of Lemma \ref{lem:roger}. 
Let $S(X_1,\dots,X_4;X)$ denote the contribution to the right hand side from
the dyadic ranges
\begin{equation}
\label{eq:reg'}
X_i/2<|x_i|\leq X_i, \quad \text{for $1\leq i\leq 4$}.
\end{equation}
It will be convenient to put $\hat X=X_1\dots X_4$.
Writing $s=\Dbad(\x)$, which is a square-full integer, we conclude that 
$$
S(X_1,\dots,X_4;X)
\ll X^{-3}\hat X^{-1/4}\sum_{\text{$s$ square-full}} s^{3/10} 
\sum_{\substack{d_1,\dots,d_4\\ d_1\dots d_4=s}} \sum_{\x\in S} L(1,\chi),
$$
where $S$ is the set of $\x\in \ZZ^4$ in the region \eqref{eq:reg'} for which
with $\Delta(\x)\neq \square$ and $d_i\mid x_i$ for $1\leq i\leq 4$.
Appealing to \eqref{lem:L}, it follows that 
\begin{align*}
S(X_1,\dots,X_4;X)
&\ll X^{-3}\hat X^{3/4} \sum_{\text{$s$ square-full}} s^{3/10} 
\sum_{\substack{d_1,\dots,d_4\\ d_1\dots d_4=s}} (d_1\dots d_4)^{-7/8}\\
&\ll X^{-3}\hat X^{3/4} \sum_{\text{$s$ square-full}} \tau_4(s)s^{-23/40} \\
&\ll X^{-3}\hat X^{3/4}.
\end{align*}
On summing over dyadic values for the $X_i$ subject to
$\max X_i\ll X$, we finally conclude that 
\beql{SB}
S(X)\ll 1. 
\eeq
Once inserted into \eqref{eq:waypoint}, this therefore completes the
proof of Lemma \ref{N56}.
\end{proof} 

\subsection{The counting function $N(\Omega;B)$} 
  
Using Lemma \ref{N3A}, together with a dyadic subdivision of the range
for $|\y|$, we find that
\begin{equation}\label{N4C}
\begin{split}
\card\Bigg\{(\x,\y)\in &\ZZp^4\times\ZZp^4:
\begin{array}{l}
\Delta(\x)\not=\square,    ~F(\x;\y)=0\\
|\x|^3|\y|^2\le B,\,|\y|\le B^{1/4}
\end{array}{}\Bigg\}\\
=~&\frac{B}{\zeta(3)}
\sum_{\substack{\y\in\ZZp^4\\ |\y|\leq B^{1/4}}}\frac{\rho_\infty(\y)}{|\y|^2}+O(B). 
\end{split}\end{equation}

We would like to handle the range $|\x|\le B^{1/6}$ similarly, using
Lemma \ref{N56}.  We claim that
\begin{equation}\label{N5C}
\begin{split}
\card\Bigg\{(\x,\y)\in&\ZZp^4\times\ZZp^4:
\begin{array}{l}
\Delta(\x)\not=\square,
    ~F(\x;\y)=0\\
|\x|^3|\y|^2\le B,\,|\x|\le B^{1/6}
\end{array}{}
\Bigg\}\\
=~&
\frac{B}{\zeta(2)}\sum_{\substack{\x\in \ZZp^4\\ B^{2\eta}\leq |\x|\leq B^{1/6}\\
  \Delta(\x)\not=\square}}
\frac{\sigma_\infty(\x)\mathfrak{S}(\x)}{|\x|^3}\\
&+O(\eta^{1/2}B\log B)+O_{\eta}(B^{1-\eta^2/20}\log B).  
\end{split}
\end{equation}
In order to prove this we must handle the contribution
of the two ranges $|\x|<B^{2\eta}$ and
$B^{1/6-4\eta}<|\x|\le B^{1/6}$, both for the number of
solutions to $F(\x;\y)=0$, and for the second range 
in respect of the sum of leading terms.
Lemma \ref{lem:stage1} shows that $N_5(B;X)\ll B$ when $X\ll B^{1/6}$,
so that the two awkward ranges contribute
$O(\eta B\log B)$ on the left, which is 
dominated by the error term $O(\eta^{1/2}B\log B)$ in (\ref{N5C}). In
view of Lemma \ref{SIB}, a range $X<|\x|\le 2X$ contributes $O(BS(X))$
to the main term on the right in (\ref{N5C}), in the notation of 
\eqref{SD*}.
  Using the bound (\ref{SB}), and summing
over dyadic values of $X$ in the  range
$B^{1/6-4\eta}\ll X\ll B^{1/6}$ 
$B^{1/6-3\eta}\ll X\ll B^{1/6}$ 
we obtain a contribution $O(\eta B\log
B)$, which again is satisfactory.  This establishes the claim in
(\ref{N5C}).

We now combine the estimates (\ref{N4C}) and (\ref{N5C}) so as to
cover the entire range $|\x|^3|\y|^2\le B$ in the definition 
(\ref{NXTB}) of the counting function  $N(\Omega;B)$, with
$\Omega=X(\QQ)\setminus T$ and $T$ being given by \eqref{eq:TT}.
We may remove the points with
$|\x|\le B^{1/6}$ and $|\y|\le B^{1/4}$ at a cost $O(B)$, using Lemma
\ref{lem:stage1}.  
Passing to the affine cone and allowing 
 for multiplication of $\x$ and $\y$ by units, 
we therefore reach the following conclusion.

\begin{lemma}\label{together}
  We have
  \begin{align*}
    N(\Omega;B)=~&
  \frac{B}{4}\left( \frac{1}{\zeta(3)}M_1(B) +
  \frac{1}{\zeta(2)}M_2(B) \right)\\
&  +O(\eta^{1/2}B\log B)+O_{\eta}(B^{1-\eta^2/20}\log B).  
\end{align*}
  where
  \begin{align*}
M_1(B)&=\sum_{\substack{\y\in \ZZp^4\\ |\y|\leq B^{1/4}}}
  \frac{\rho_\infty(\y)}{|\y|^2} \quad \text{ and } \quad
M_2(B)=  
  \sum_{\substack{\x\in \ZZp^4\\ B^{2\eta}\leq |\x|\leq B^{1/6}\\
  \Delta(\x)\not=\square}}
\frac{\sigma_\infty(\x)\mathfrak{S}(\x)}{|\x|^3}.
  \end{align*}
  \end{lemma}

\section{The final reckoning}\label{s:reckon}

In this  section we shall produce asymptotic formulae for $M_1(B)$
and $M_2(B)$, as $B\to \infty$.  We shall begin in \S 
\ref{s:analyseM1} by dealing with $M_1(B)$, which is the easier
to handle, before developing the techniques further in \S
\ref{s:analyseM2} to handle $M_2(B)$. 
Finally, in \S \ref{s:final} we shall
confirm that the two contributions combine in a satisfactory manner
to complete the proof of Theorem \ref{t:main}.

\subsection{Analysis of $M_1(B)$}\label{s:analyseM1}

The goal of the present section is the following result.
\begin{lemma}\label{M1B}
  We have
  \[M_1(B)=\frac{1}{2\zeta(4)}\tau_\infty \log B+O(1),\]
where $\tau_\infty$ is given by \eqref{eq:tau_infty}. 
  \end{lemma}

We begin by using the M\"{o}bius function to detect the primitivity
condition, which shows that
\[M_1(B)=\sum_{k\leq B^{1/4}}\frac{\mu(k)}{k^2}
\sum_{\y\in \ZZ^4\cap T_0}\frac{\rho_\infty(k\y)}{|\y|^2},\]
  where
  \beql{Y0D}
  T_0=T_0(k)=   \{\y\in\RR^4:1\le|\y|\le B^{1/4}/k\}.
  \eeq
One sees from the definition (\ref{eq:def-rho}) that
$\rho_{\infty}(k\y)=k^{-2}\rho_{\infty}(\y)$, whence
\beql{M1BF}
M_1(B)=\sum_{k\leq B^{1/4}}\frac{\mu(k)}{k^4}
\sum_{\y\in \ZZ^4\cap T_0}\frac{\rho_\infty(\y)}{|\y|^2}.
\eeq

We now wish to replace the sum over $\y$ by an integral. The argument
will make repeated use of the bound
\beql{SIB1}
\rho_\infty(\y)\ll|\y|^{-2},
\eeq
which is immediate from (\ref{eq:2constants}). 
We start with the following estimate.

\begin{lemma}\label{sumint1}
  If $\min_i|y_i|\ge 2$  then
\[\frac{\rho_\infty(\y)}{|\y|^2}=
\int_{[0,1]^4}\frac{\rho_\infty(\y+\t)}{|\y+\t|^2}\d\t+
O\big((\min_i|y_i|)^{-1/3}|\Delta(\y)|^{-2/3}|\y|^{-2}\big).\]
\end{lemma}
\begin{proof}
  We begin by showing that
  \beql{ccc1}
  \nabla\rho_\infty(\u)\ll
  (\min_i|u_i|)^{-1/3}|\Delta(\u)|^{-2/3}.
  \eeq
Without loss of generality we may just examine the partial derivative with
respect to $y_1$. Our definition (\ref{eq:def-rho}) shows that
\[\rho_\infty(\u)=\int_{-\infty}^{\infty}\prod_{i=1}^4
I(-\theta u_i^2) \d \theta,\]
where we write temporarily
\[I(\psi)=\int_{-1}^1 e(\psi x) \d x.\]
Then $I(\psi)\ll\min\{1,|\psi|^{-1}\}$ and
\[\frac{\partial }{\partial u_1}I(-\theta u_1^2)=-4\pi i\theta u_1
\int_{-1}^1 xe(-\theta xu_1^2)\d x\ll |u_1|^{-1}.\]
Thus
\begin{align*}
\frac{\partial }{\partial u_1} 
\rho_\infty(\u) 
&\ll  |u_1|^{-1}
\int_{-\infty}^\infty \min\{1,|\theta|^{-3} |u_2u_3u_4|^{-2}\}
\d \theta\\
&\ll  |u_1|^{-1}|u_2u_3u_4|^{-2/3}\\
&\ll  (\min_i|u_i|)^{-1/3}|\Delta(\u)|^{-2/3},
\end{align*}
as required.

We now use the decomposition
\[\frac{\rho_\infty(\y+\t)}{|\y+\t|^2}-
\frac{\rho_\infty(\y)}{|\y|^2}=
\frac{\rho_\infty(\y+\t)-\rho_{\infty}(\y)}{|\y+\t|^2}+
\rho_\infty(\y)\{|\y+\t|^{-2}-|\y|^{-2}\}.\]
If $|\t|\le 1$ and $\min_i|y_i|\ge 2$ then 
\[|y_i+t_i|\ge|y_i|-|t_i|\ge\tfrac12 |y_i|,\] 
so that $|\y+\t|^{-2}\ll|\y|^{-2}$.  Moreover the Mean Value Theorem
shows that
\[|\rho_\infty(\y+\t)-\rho_{\infty}(\y)|\le
\sup_{0\le\xi\le 1}\left|\frac{\partial}{\partial\xi}
\rho_\infty(\y+\xi\t)\right|.\]
It then follows from (\ref{ccc1}) that
\[\frac{\rho_\infty(\y+\t)-\rho_{\infty}(\y)}{|\y+\t|^2}\ll
(\min_i|y_i|)^{-1/3}|\Delta(\y)|^{-2/3}|\y|^{-2}.\]
We also have
\[|\y+\t|^{-2}-|\y|^{-2}=
|\y+\t|^{-2}|\y|^{-2}\{|\y+\t|+|\y|\}\{|\y|-|\y+\t|\}.\]
Assuming as above that $|\t|\le 1$ and $|\y|\ge 2$ we see that
$|\y|\ll|\y+\t|\ll|\y|$ and $|\y+\t|-|\y|\ll 1$, so that
\[\rho_\infty(\y)\{|\y+\t|^{-2}-|\y|^{-2}\}\ll
|\y|^{-3}\rho_\infty(\y)\ll
(\min_i|y_i|)^{-1/3}|\Delta(\y)|^{-2/3}|\y|^{-2},\]
by (\ref{SIB1}).  We therefore have
\beql{extra}
\frac{\rho_\infty(\y+\t)}{|\y+\t|^2}=
\frac{\rho_\infty(\y)}{|\y|^2}+
O\left((\min_i|y_i|)^{-1/3}|\Delta(\y)|^{-2/3}|\y|^{-2}\right),
\eeq
and the lemma follows.
\end{proof}

Our next result converts the summation over $\y$ in (\ref{M1BF})
into an integral.
 \begin{lemma}\label{M1E2}
  We have
\[  
\sum_{\y\in \ZZ^4\cap T_0}
\frac{\rho_\infty(\y)}{|\y|^2}=J_1(B;k)+O(1),\]
where
\[J_1(B;k)=\int_{T_0(k)}\frac{\rho_\infty(\y)}{|\y|^2}\d\y.\]
\end{lemma}
\begin{proof}
We define
\[X=\{\y\in \ZZ^4: |\y|\leq B^{1/4}/k-2,\,\min|y_i|\ge 2\}\]
and
\[Y=\bigcup_{\y\in X}(\y+(0,1]^4).\] 
The reader should note that these could be empty if $k$ is
large enough.
The sets $\y+(0,1]^4$ 
forming $Y$ are disjoint, and $Y$ 
  lies inside the set $T_0$ defined in (\ref{Y0D}).  
Moreover $T_0\setminus Y$ is a subset of
$T_1\cup T_2$, where
 \[T_1 = \{\t\in T_0: B^{1/4}/k-3\leq |\t|\leq B^{1/4}/k\},\]
and
\[T_2=\{\t\in T_0: \min|t_i|\leq 3\}.\]
It then follows from Lemma \ref{sumint1} and (\ref{SIB1}) that
$$
\sum_{\y\in \ZZ^4\cap T_0} \frac{\rho_\infty(\y)}{|\y|^2}=
J_1(B;k)+O\left(\sum_{i=0}^2 E_i\right), 
$$
where
\[E_0=\sum_{\substack{\y\in \ZZ^4\cap T_0\\ \min |y_i|\geq 2}}  
(\min_i|y_i|)^{-1/3}|\Delta(\y)|^{-2/3}|\y|^{-2},\]
and
\[E_i=\sum_{\y\in \ZZ^4\cap T_i} |\y|^{-4}+\int_{T_i}|\y|^{-4}\d\y\]
for $i=1,2$. We readily find that $E_i\ll 1$ for $i=0,1,2$, and the
lemma follows.
\end{proof}

In order to complete our argument we will need the following
evaluation of $J_1(B;k)$.
\begin{lemma}\label{JE}
  If $k\le B^{1/4}$ we have
  \[J_1(B;k)=2\tau_{\infty}\log (B^{1/4}/k),\]
  with $\tau_{\infty}$ given by \eqref{eq:tau_infty}.
\end{lemma}
\begin{proof}
We divide $T_0(k)$ into four (overlapping) pieces according to the
index $i$ for which $|\y|=|y_i|$.  We observe from (\ref{eq:def-rho})
that $\rho_{\infty}(\y)$ is unchanged when we permute the coordinates,
and that $\rho_{\infty}(\y)=|\y|^{-2}\rho_{\infty}(t_1,t_2,t_3,1)$
if $|\y|=|y_4|$ and
$t_i=y_i/|\y|$ for $i=1,2,3$.  It therefore follows that
\[J_1(B;k)=8\int_1^{B^{1/4}/k}\frac{\d y_4}{y_4}
\int_{[-1,1]^3}\rho_\infty(t_1,t_2,t_3,1)\d\t.\]
In a precisely similar way Lemma \ref{tinf} yields
\[\tau_{\infty}=\int_{[-1,1]^4}\rho_{\infty}(\y)\d\y=
8\int_0^{1}y_4\d y_4
\int_{[-1,1]^3}\rho_\infty(t_1,t_2,t_3,1)\d\t,\]
and the lemma follows.
\end{proof}

We can now complete the proof of Lemma~\ref{M1B}. Combining
(\ref{M1BF}) with Lemma~\ref{M1E2}
we obtain
\[M_1(B)=\sum_{k\leq B^{1/4}}\frac{\mu(k)}{k^4}J_1(B;k)+O(1).\]
We have
\[\sum_{k\leq B^{1/4}}\frac{\mu(k)}{k^4}\log k\ll 1\]
and
\[\sum_{k\leq B^{1/4}}\frac{\mu(k)}{k^4}=\zeta(4)^{-1}+O(B^{-3/4}),\] 
so that Lemma \ref{JE} yields
\[M_1(B)=\frac{2\tau_{\infty}}{\zeta(4)}\log B^{1/4}+O(1),\]
and the required estimate follows.

\subsection{Analysis of $M_2(B)$}\label{s:analyseM2}

We remind the reader that
\[M_2(B)=  
  \sum_{\substack{\x\in \ZZp^4\\ B^{2\eta}\leq |\x|\leq B^{1/6}\\
  \Delta(\x)\not=\square}}
  \frac{\sigma_\infty(\x)\mathfrak{S}(\x)}{|\x|^3}\]
  where
\[  \sigma_\infty(\x)=\int_{-\infty}^{\infty}\int_{[-1,1]^4}
e(-\theta F(\x;\y))\d\y\d\theta\]
and
\[\mathfrak{S}(\x)=\sum_{q=1}^\infty q^{-4}S_q, \quad \text{with }
S_q=S_q(\x)= \sumstar_{\substack{a\bmod{q}}}~
\sum_{\b\bmod{q}}  e_q \left(aF(\ma{x};\ma{b})\right).\]
The goal of the present section is the following result.
\begin{lemma}\label{M2P}
  We have
  \[M_2(B)= 
\frac{1}{2}\cdot \frac{\zeta(2)}{\zeta(3)\zeta(4)} \cdot \tau_\infty \log B
+ O\left(\eta \log B\right) +O_\eta(1),\]
where $\tau_\infty$ is given by \eqref{eq:tau_infty}.
\end{lemma}

In order to estimate
$M_2(B)$ our plan will begin by showing that the 
singular series $\mathfrak{S}(\x)$ can be  replaced by a truncated sum
$$
\mathfrak{S}(\x;R)=\sum_{q\leq R}  q^{-4} S_q,
$$
for suitable $R$. Using Heath-Brown's large sieve for real
characters \cite{HB95},  
we shall ultimately succeed in showing that $R$ can be taken an
arbitrarily small power of $B$, with acceptable error.
The  constraint $\Delta(\x)\neq \square$ can now be replaced by
$\Delta(\x)\not=0$, again with an acceptable error. 
We then interchange the $q$ and $\x$ summations
in $M_2(B)$ and approximate the $\x$-sum
by a $4$-fold integral.
Lastly, the remaining $q$-sum will be extended to infinity to
get our final asymptotic formula for $M_2(B)$. Throughout this
analysis we will use repeatedly the estimate
\[\sigma_{\infty}(\x)\ll|\Delta(\x)|^{-1/4}\] 
given by Lemma \ref{SIB}.

In order to carry out this plan we first make a crude first analysis of 
the tail of the singular series 
$\mathfrak{S}(\x)$.  
Since  $\Delta(\x)\neq \square$, 
it follows from  Lemma \ref{lem:average2} and partial summation that 
$$
\sum_{q> B}q^{-4}S_q\ll_\ve    
|\Delta(\x)| ^{3/16} 
 \Dbad(\x)^{3/8} B^{-1/2+\ve},
 $$
 for any $\ve>0$.
Since $|\x|^3\geq |\Delta(\x)|^{3/4}$, 
an application of Lemma \ref{SIB} now shows that 
the tail of the singular series contributes
\[\ll_\ve B^{-1/2+\ve} \sum_{\substack{\x\in \ZZp^4\\  |\x|\leq B^{1/6}}}
\frac{\Dbad(\x)^{3/8}}{|\Delta(\x)|^{13/16}}\]
to $M_2(B)$. Writing $s=\Dbad(\x)$ and $n=|\Delta(\x)|$, we see that this is
\begin{align*}
\ll_\ve B^{-1/2+\ve} 
\sum_{\substack{s\leq B^{2/3}\\ \text{$s$ square-full}}} s^{3/8}
\sum_{\substack{ n\leq B^{2/3}\\ s\mid n}} \frac{\tau_4(n)}{n^{13/16}}
&\ll_\ve B^{-3/8+2\ve} 
\sum_{\substack{\text{$s$ square-full}}}\frac{1}{s^{5/8}}.
\end{align*}
Taking $\ve=\frac{1}{8}$ and noting that 
the $s$-sum is convergent, this shows that 
\[M_2(B)= \sum_{\substack{\x\in \ZZp^4\\ B^{2\eta}\leq |\x|\leq
    B^{1/6}\\ \Delta(\x)\not=\square}} 
\frac{\sigma_\infty(\x) \mathfrak{S}(\x;B)}{|\x|^3}
+ O\left(B^{-1/8}\right).\] 
Building on this, we now show that the singular series can be truncated
to a much smaller power of $B$, with acceptable error.

\begin{lemma}\label{61}
We have 
$$
M_2(B)= \sum_{\substack{\x\in \ZZp^4\\ B^{2\eta}\leq |\x|\leq
    B^{1/6}\\ \Delta(\x)\not=\square}} 
\frac{\sigma_\infty(\x) \mathfrak{S}(\x;B^{\eta/8})}{|\x|^3}
+ O_\eta(1).
$$
\end{lemma}

\begin{proof}
  Since $|\x|^3\geq |\Delta(\x)|^{3/4}$ and
  $\sigma_\infty(\x)\ll |\Delta(\x)|^{-1/4}$ it will be enough to
  show that
  \[E(B)
  =  \sum_{\substack{\x\in \ZZp^4\\ B^{2\eta}\leq |\x|\leq B^{1/6}
  \\ \Delta(\x)\not=0}} \frac{1}{|\Delta(\x)|}
\left| \sum_{B^{\eta/8} <q\leq B} q^{-4}S_q\right|\ll_{\eta} 1.\]
Note that we have relaxed the condition $\Delta(\x)\not=\square$ to
require only that $\Delta(\x)$ is non-zero.  We shall write $s=\Dbad(\x)$
and $n=\Delta(\x)$, so that $ns^{-1}$ is square-free. Since 
$|\Delta(\x)|\geq |\x|\geq B^{2\eta}$,  we are only interested in
integers $n$ in the range $B^{2\eta}\leq |n|\leq B^{2/3}$.

It follows from the multiplicativity of $S_q$ that 
\[\left|\sum_{\substack{B^{\eta/8}<q\leq {B}}} q^{-4}S_{q}\right|
\leq 
\sum_{\substack{v\leq {B}\\ v\mid (2s)^\infty}}v^{-4}|S_{v}|\left|
\sum_{\substack{B^{\eta/8}/v<u\leq B/v\\ (u,2s)=1}}u^{-4}S_{u}\right|. \]
Lemma \ref{lem:Sq-basic} shows that 
\begin{align*}
  S_v&\ll v^3 \prod_{1\leq i\leq 4} (v,x_i)^{1/2}\\
  &\leq v^3(v^4,x_1\dots x_4)^{1/2}\\
  &\ll v^3(v^4,s)^{1/2}\\  
&\leq v^3\min(v^4,s)^{1/2}\\
&\leq v^{7/2}s^{3/8} 
\end{align*}
for $v\mid (2s)^\infty$.
Moreover, 
Lemmas \ref{lem:good_primes} and \ref{lem:zero-c} imply that 
\begin{align*}
S_{u}=
\left(\frac{n}{u}\right) \phi^*(u) u^3,
\end{align*}
when $(u,2s)=1$, where $\phi^*=1*h$, with $h(d)=\mu(d)/d$.
Hence
\[E(B)\ll_\ve B^\ve \hspace{-0.2cm}
\sum_{\substack{s\leq B^{2/3}\\\text{$s$ square-full}}}
\sum_{\substack{v\leq B\\ v\mid (2s)^\infty}} v^{-1/2}s^{3/8} 
\sum_{\substack{B^{2\eta}\leq |n|\leq B^{2/3}\\ s\mid n}} \frac{1}{|n|}
\left|\sum_{\substack{B^{\eta/8}/v<u\leq {B}/v\\ (u,2s)=1}} 
\left(\frac{n}{u}\right) \frac{\phi^*(u)}{u}\right|,\]
since the number of $\x$ associated to $n$ is at most
$\tau_4(n)=O_\ve(B^\ve)$.  

We now write $n=sm$ and split the ranges for $m$ and $u$ into dyadic
intervals.  This gives us values 
$M$ and $U\le U_1\le 2U$, with
$$ \max\left(1\,,\,\frac{B^{2\eta}}{s}\right)\ll M\ll B^{2/3}\;\;\;
\mbox{and}\;\;\;\frac{B^{\eta/8}}{v}\ll U\ll\frac{B}{v}
$$
such that
\[E(B)\ll_\ve \frac{B^{2\ve}}{MU}
\sum_{\substack{s\leq B^{2/3}\\\text{$s$ square-full}}}
\sum_{\substack{v\leq B\\ v\mid (2s)^\infty}}v^{-1/2}s^{-5/8} 
\sum_{M<m\le 2M}
\left|\sum_{U<u\le U_1}\left(\frac{m}{u}\right)\alpha_{u,s}\right|,\]
with 
\[\alpha_{u,s}=
\begin{cases}
U\frac{\phi^*(u)}{u}\left(\frac{4s}{u}\right) &\text{ if $u$ is odd,}\\
0 &\text{ if $u$ is even.}
\end{cases}\]
In particular $\alpha_{u,s}\ll 1$.
We now write
\[\sum_{M<m\le 2M}
\left|\sum_{U<u\le U_1}\left(\frac{m}{u}\right)\alpha_{u,s}\right|=
\sum_{M<m\le 2M}\sum_{U<u\le U_1}\left(\frac{m}{u}\right)\alpha_{u,s}\beta_m,\]
with $\beta_m=\pm 1$, and apply the large sieve for real characters in
the form given by Heath-Brown \cite[Cor.~4]{HB95}.  This shows that
\[\sum_{M<m\le 2M}\sum_{U<u\le U_1}\left(\frac{m}{u}\right)\alpha_{u,s}\beta_m
\ll_{\ve} (MU)^{\ve}\{MU^{1/2}+M^{1/2}U\},\]
whence
\[  E(B)\ll_\ve B^{4\ve} \sum_{\substack{s\leq B^{2/3}\\\text{$s$ square-full}}}
\sum_{\substack{v\leq B\\ v\mid (2s)^\infty}}v^{-1/2}s^{-5/8} 
\{U^{-1/2}+M^{-1/2}\}.\]
However
\begin{align*}
  U^{-1/2}+M^{-1/2}&\ll v^{1/2}B^{-\eta/16}+\min\{1\,,\, 
s^{1/2}B^{-\eta}\}\\
  &\ll v^{1/2}B^{-\eta/16}+1^{15/16}(s^{1/2}B^{-\eta})^{1/16}\\
  &\ll B^{-\eta/16}v^{1/2}s^{1/32}.
  \end{align*}
We therefore deduce that
\[  E(B)\ll_\ve B^{-\eta/16+4\ve}\sum_{\substack{s\leq B^{2/3}\\\text{$s$ square-full}}}
s^{-19/32} \sum_{\substack{v\leq B\\ v\mid (2s)^\infty}}1 \ll_\ve B^{-\eta/16+5\ve},\]
since there are $O_{\ve}((sB)^{\ve})$ possible values for $v$, and
the sum over square-full $s$ is convergent.
Taking $\ve=\eta/80$,
we therefore conclude the proof of the lemma. 
\end{proof}

Next we wish to show that the condition $\Delta(\x)\not=\square$
can be replaced by $\Delta(\x)\not=0$ with an acceptable error. We
trivially have
\[|\mathfrak{S}(\x;B^{\eta/8})|\le\sum_{q\le B^{\eta/8}}q\ll B^{\eta/4}.\]
Moreover Lemma \ref{SIB} shows that
$\sigma_\infty(\x)\ll|\Delta(\x)|^{-1/4}$.  We now write
$\Delta(\x)=n^2$ so that $|\x|\le n^2\le|\x|^4$. In particular we
will have $B^{\eta}\leq n\leq B^{1/3}$ and $|\x|^{-3}\le n^{-3/2}$.
Moreover each value
of $n$ corresponds to $O_{\ve}(B^{\ve})$ vectors $\x$.  Thus
\[\sum_{\substack{\x\in \ZZp^4\\ B^{2\eta}\leq |\x|\leq B^{1/6}\\
  \Delta(\x)=\square\not=0}}
\frac{\sigma_\infty(\x) \mathfrak{S}(\x;B^{\eta/8})}{|\x|^3}
\ll_{\ve} B^{\eta/4+\ve}\sum_{B^{\eta}\leq n\leq B^{1/3}}n^{-2}
\ll_{\ve} B^{-3\eta/4+\ve}.\]
This may be absorbed into the error term of Lemma \ref{61} on choosing
$\ve=3\eta/4$.

Now that we have truncated the singular series satisfactorily, we may
open up the expression for 
$\mathfrak{S}(\x;B^{\eta/8})$ and interchange the $q$-sum with
the $\x$-sum, before breaking the latter
into congruence classes modulo $q$.  This leads to the expression
$$
M_2(B)=\sum_{q\leq B^{\eta/8}}q^{-4}
\sum_{\substack{\a,\b \bmod {q}\\ (q,\a)=1}} c_q(F(\a;\b))
U(q;\a)+O_\eta(1),
$$
where $c_q(\cdot)$ is the Ramanujan sum and
\[U(q;\a)=
\sum_{\substack{\x\in \ZZp^4\\ B^{2\eta}\leq |\x|\leq B^{1/6}\\
\Delta(\x)\not=0 \\     \x\equiv\a \bmod{q}}} 
\frac{\sigma_\infty(\x) }{|\x|^3}.\]

Any $\x$ counted by $U(q;\a)$ is automatically coprime to $q$.
Using the M\"obius function to detect the residual primitivity of $\x$ 
we may now write
\[U(q;\a)=\sum_{\substack{k\leq B^{1/6}\\ (k,q)=1}} \frac{\mu(k)}{k^3}
\sum_{\substack{\x\in \ZZ^4\cap T_0\\
  \x\equiv\bar k\a \bmod{q}}}\frac{\sigma_\infty(k\x) }{|\x|^3},\]
where $\bar k$ is the multiplicative inverse of $k$ modulo $q$ and
\beql{T0D}
T_0=T_0(k)
=\{\t\in (\RR_{\not=0})^4: B^{2\eta}/k\leq |\t|\leq B^{1/6}/k\}. 
\eeq
(The reader should note that this is not the same set that is defined in
(\ref{Y0D});  we recycle our notation for this and other similar sets.)
Since $\sigma_\infty(k\x)=k^{-1}\sigma_\infty(\x)$, this
simplifies to give
$$
U(q;\a)=\sum_{\substack{k\leq B^{1/6}\\ (k,q)=1}} \frac{\mu(k)}{k^4}
\sum_{\substack{\x\in \ZZ^4\cap T_0\\
\x\equiv\bar k\a \bmod{q}}}\frac{\sigma_\infty(\x)}{|\x|^3}.
$$
We therefore obtain the following formula.
\begin{lemma}\label{M2E1}
  We have
\[M_2(B)=\sum_{q\leq B^{\eta/8}}q^{-4}
\sum_{\substack{k\leq B^{1/6}\\ (k,q)=1}} \frac{\mu(k)}{k^4}
\sum_{\substack{\a,\b \bmod {q}\\ (q,\a)=1}} c_q(F(\a;\b))
\sum_{\substack{\x\in \ZZ^4\cap T_0\\
\x\equiv\bar k\a \bmod{q}}}\frac{\sigma_\infty(\x)}{|\x|^3}+O_{\eta}(1).\]
\end{lemma}

The next stage is to compare the $\x$-sum to an integral.
We start with the following estimate, which is an analogue of Lemma
\ref{sumint1}. 
\begin{lemma}\label{sumint}
  If $\min_i|x_i|\ge 2q$  then
\[\frac{\sigma_\infty(\x)}{|\x|^3}=
q^{-4}\int_{[0,q]^4}\frac{\sigma_\infty(\x+\t)}{|\x+\t|^3}\d\t
+O(q(\min_i|x_i|)^{-1}|\Delta(\x)|^{-1/4}|\x|^{-3}).\]
\end{lemma}
\begin{proof}
  The proof follows the same lines as that of Lemma \ref{sumint1}.
  We begin by showing that
  \beql{ccc}
     \nabla\sigma_\infty(\x)\ll(\min_i|x_i|)^{-1}|\Delta(\x)|^{-1/4}.
   \eeq
Without loss of generality it will suffice to examine the partial
derivative with respect to 
$x_1$  
in   proving this.
Our definition of the singular integral yields
\[\sigma_\infty(\x)=\int_{-\infty}^{\infty}\prod_{i=1}^4
I(-\theta x_i) \d \theta,\]
where we write temporarily
\[I(\psi)=\int_{-1}^1 e(\psi y^2) \d y.\]
By the standard second derivative test \cite[Lemma  4.4]{titch} we see that
\[I(\psi)\ll \min\{1, |\psi|^{-1/2}\}.\]
Moreover
\[\frac{\partial }{\partial x_1}I(-\theta x_1)=
\frac{1}{2x_1}\int_{-1}^1 y\frac{\partial}{\partial y}e(-\theta x_1 y^2)\d y.\]
The integral on the right is uniformly bounded, as one sees on integrating by
parts. Thus
\[\frac{\partial }{\partial x_1}I(-\theta x_1)\ll |x_1|^{-1},\]
whence
\begin{align*}
\frac{\partial }{\partial x_1} \sigma_\infty(\x)
&\ll  |x_1|^{-1}
\int_{-\infty}^\infty \min\{1,|\theta|^{-3/2} |x_2x_3x_4|^{-1/2}\}
\d \theta\\
&\ll  |x_1|^{-1}|x_2x_3x_4|^{-1/3}\\
&\ll  
(\min_i|x_i|)^{-1}
|\Delta(\x)|^{-1/4},
\end{align*}
as required.

We now use the decomposition
\begin{align*}
\frac{\sigma_\infty(\x+\t)}{|\x+\t|^3}&-
  \frac{\sigma_\infty(\x)}{|\x|^3}\\
&=
\frac{\sigma_\infty(\x+\t)-\sigma_{\infty}(\x)}{|\x+\t|^3}+
\sigma_\infty(\x)\{|\x+\t|^{-3}-|\x|^{-3}\}.
\end{align*}
If $|\t|\le q$ and $\min_i|x_i|\ge 2q$ we may show, as in the proof 
of (\ref{extra}), that
\begin{eqnarray*}
\frac{\sigma_\infty(\x+\t)-\sigma_{\infty}(\x)}{|\x+\t|^3}&\ll&
 |\x|^{-3}q\sup_{0\le\xi\le 1}\left|\frac{\partial}{\partial\xi}
 \sigma_\infty(\x+\xi\t)\right|\\
&\ll& q(\min_i|x_i|)^{-1}|\Delta(\x)|^{-1/4}|\x|^{-3},
\end{eqnarray*}
via (\ref{ccc}).  Similarly, also as in the proof of (\ref{extra}), we
will have
\[|\x+\t|^{-3}-|\x|^{-3}\ll q|\x|^{-4},\]
so that Lemma \ref{SIB} yields
\[\sigma_\infty(\x)\{|\x+\t|^{-3}-|\x|^{-3}\}\ll
q|\Delta(\x)|^{-1/4}|\x|^{-4}.\]
We therefore have
\[\frac{\sigma_\infty(\x+\t)}{|\x+\t|^3}=
\frac{\sigma_\infty(\x)}{|\x|^3}+
O(q(\min_i|x_i|)^{-1}|\Delta(\x)|^{-1/4}|\x|^{-3}),\]
and the lemma follows.
\end{proof}

We are now ready to tackle the $\x$-summation in Lemma \ref{M2E1}.
 \begin{lemma}\label{M2E2}
  We have
\[  \sum_{\substack{\x\in \ZZ^4\cap T_0\\
    \x\equiv\bar k\a \bmod{q}}}\frac{\sigma_\infty(\x)}{|\x|^3}=
q^{-4}J_2(B;k)+O(qkB^{-3\eta/2}),\]
where
\[J_2(B;k)
=\int_{T_0(k)}\frac{\sigma_\infty(\y) }{|\y|^3}\d\y.\]
\end{lemma}
\begin{proof}
We define
\[X=\{\x\in \ZZ^4: B^{2\eta}/k+2q\leq |\x|\leq B^{1/6}/k-2q,\,
\min|x_i|\ge 2q,\, \x\equiv\bar k\a\bmod{q}\}\]
and
\[Y=\bigcup_{\x\in X}(\x+(0,q]^4).\] 
The reader should note that these could be empty if $k$ and $q$
are large enough.
The sets $\x+(0,q]^4$ forming $Y$ are disjoint, and both $X$ and $Y$ 
lie inside the set $T_0$ defined in (\ref{T0D}).
Moreover $T_0\setminus Y$ is a subset of 
$T_1\cup T_2\cup T_3$, where
\[T_1
=\{\t\in T_0: B^{2\eta}/k\leq |\t|\leq B^{2/\eta}/k+3q\},\]
 \[T_2
 =\{\t\in T_0: B^{1/6}/k-3q\leq |\t|\leq B^{1/6}/k\},\]
and
\[T_3
= \{\t\in T_0: \min|t_i|\leq 3q\}.\]
It then follows from Lemma \ref{sumint} that
\beql{IqaE}
\sum_{\substack{\x\in \ZZ^4\cap T_0\\ \x\equiv\bar k\a \bmod{q}}}
\frac{\sigma_\infty(\x)}{|\x|^3}=q^{-4}J_2(B;k)+O\left(\sum_{i=0}^3 E_i\right),
\eeq
where
\[E_0=q\sum_{\substack{\x\in \ZZ^4\cap T_0\\ \min|x_i|\geq 2q }}
(\min_i|x_i|)^{-1}|\Delta(\x)|^{-1/4}|\x|^{-3},\]
and
\[E_i=\sum_{\substack{\x\in \ZZ^4\cap T_i }}
|\Delta(\x)|^{-1/4}|\x|^{-3}+
q^{-4}\int_{T_i}|\Delta(\y)|^{-1/4}|\y|^{-3}\d\y\]
for $i=1,2,3$. 
Note that we have dropped the condition
$\x\equiv\bar k\a \bmod{q}$ in these error terms. We now find that 
\[E_0\ll q
\hspace{-0.1cm}
\sum_{\substack{2q \leq x_1\le x_2,x_3\le x_4\\ x_4\ge B^{2\eta}/k}}
    \hspace{-0.1cm}
x_1^{-5/4}(x_2x_3)^{-1/4}x_4^{-13/4}\ll
q^{3/4}\hspace{-0.1cm}
\sum_{x_4\ge B^{2\eta}/k}
\hspace{-0.1cm} x_4^{-7/4}\ll (qkB^{-2\eta})^{3/4},\] 
for example.  This holds whether $k\le B^{2\eta}$ or not.
Similar calculations show that
\[E_1\ll
\sum_{B^{2\eta}/k\le x_4\le B^{2\eta}/k+3q}x_4^{-1}
+\int_{B^{2\eta}/k}^{B^{2\eta}/k+3q}y_4^{-1}\d y_4
\ll qkB^{-2\eta},\]
and
\[E_2\ll qkB^{-1/6}.\]
For the sum in $E_3$ we have
\begin{align*}
    \sum_{\x\in T_3} 
|\Delta(\x)|^{-1/4}|\x|^{-3}&\ll\sum_{\substack{1\le x_1\le 3q\\ 
 1\le x_2,x_3\le x_4\\ x_4\ge B^{2\eta}/k}}(x_1x_2x_3)^{-1/4}x_4^{-13/4}\\
&\ll q^{3/4}\sum_{x_4\ge B^{2\eta}/k}x_4^{-7/4}\\
&\ll (qkB^{-2\eta})^{3/4}, 
\end{align*}
and similarly for the integral.  Thus (\ref{IqaE}) becomes
\[\sum_{\substack{\x\in \ZZ^4\cap T_0\\ \x\equiv\bar k\a \bmod{q}}}
\frac{\sigma_\infty(\x)}{|\x|^3}=q^{-4}J_2(B;k)+O(qkB^{-3\eta/2}),\]
as required.
\end{proof}

Combining Lemma \ref{M2E2} with Lemma \ref{M2E1} we see that
\beql{C21}
M_2(B)=\sum_{q\leq B^{\eta/8}}q^{-8}\psi(q)
\sum_{\substack{k\leq B^{1/6}\\ (k,q)=1}} \frac{\mu(k)}{k^4}J_2(B;k)
+O_\eta(1),
\eeq
with
\[\psi(q)
=\sum_{\substack{\a,\b \bmod {q}\\ (q,\a)=1}} c_q(F(\a;\b)).\]
We therefore need information about the function $\psi$.
\begin{lemma}\label{pJ}
The function $\psi$ is multiplicative, with
\[  \psi(p^f)=
\begin{cases}
\phi(p^f)p^{6f}(1-p^{-4}), &\text{ if
 $ 2\mid f$},\\ 
  0, & \text{ if $2\nmid f$},  
  \end{cases}
  \]
for every positive integer $f$.
\end{lemma}

\begin{proof}
  The function $\psi(q)$ is clearly multiplicative, and for prime powers we
have
\[\psi(p^f)=\sum_{\substack{c \bmod {p^f}\\ (c,p)=1}}
\sum_{\substack{\a,\b \bmod {p^f}\\ p\nmid\a}}e_{p^f}(cF(\a;\b))=
\phi(p^f)\sum_{\substack{\a,\b \bmod{p^f}\\ p\nmid\a}}e_{p^f}(F(\a;\b)),\]
on replacing $c\a$ by $\a$.  It follows that
\[\psi(p^f)=\phi(p^f)\sum_{\a,\b \bmod {p^f}}e_{p^f}(F(\a;\b))-
\phi(p^f)\sum_{\substack{\a,\b \bmod {p^f}\\ p\mid\a}}e_{p^f}(F(\a;\b)).\]
Hence if we write
\[\psi_1(p^f)=\sum_{\a,\b \bmod {p^f}}e_{p^f}(F(\a;\b))\]
we will have $\psi(p^f)=\phi(p^f)\psi_1(p^f)-p^4\phi(p^f)\psi_1(p^{f-1})$.
However, on performing the summation over $\a$ we find that
\[\psi_1(p^f)=p^{4f}\#\{\b \bmod {p^f}:\,p^f\mid(b_1^2,\ldots,b_4^2)\}=
p^{4f}(p^{[f/2]})^4.\]
The required formula for $\psi(p^f)$ then follows.
\end{proof}

We also have the following evaluation of $J_2(B;k)$.
\begin{lemma}\label{JE2}
  We have
  \[J_2(B;k)=3\tau_{\infty}\log (B^{1/6-2\eta}).\]
\end{lemma}
\begin{proof}
The argument is completely analogous to that used for Lemma \ref{JE},
based on the fact that
$\sigma_{\infty}(\y)=|\y|^{-1}\sigma_{\infty}(t_1,t_2,t_3,1)$
if $|\y|=|y_4|$ and $t_i=y_i/|\y|$ for $i=1,2,3$.  We find that
\[J_2(B;k)=8\int_{B^{2\eta}/k}^{B^{1/6}/k}\frac{\d y_4}{y_4}
\int_{[-1,1]^3}\sigma_\infty(t_1,t_2,t_3,1)\d\t,\]
while Lemma \ref{tinf} yields
\[\tau_\infty=\int_{[-1,1]^4}\sigma_\infty(\y)\d\y=
8\int_0^1 y_4^2\d y_4
\int_{[-1,1]^3}\sigma_\infty(t_1,t_2,t_3,1)\d\t.\]
The lemma follows from these relations.
\end{proof}  

We now have everything in place to complete the proof of Lemma
\ref{M2P}.  According to Lemma \ref{JE2} we have
\begin{align*}
\sum_{\substack{k\leq B^{1/6}\\ (k,q)=1}}\frac{\mu(k)}{k^4}J_2(B;k)&=
3\tau_{\infty}\log (B^{1/6-2\eta})
\sum_{\substack{k\leq B^{1/6}\\ (k,q)=1}} \frac{\mu(k)}{k^4}\\
&=3\tau_{\infty}\log (B^{1/6-2\eta})
\sum_{\substack{k=1\\ (k,q)=1}}^{\infty}\frac{\mu(k)}{k^4}
+O(B^{-1/2}\log B)\\
&=\frac{3\tau_{\infty}\log (B^{1/6-2\eta})}{\zeta(4)}
\prod_{p\mid q}(1-p^{-4})^{-1}+O(B^{-1/2}\log B).
\end{align*}
We can now insert this into (\ref{C21}), using Lemma \ref{pJ} to
observe that $\psi(q)$ is supported on the squares, with
$\psi(r^2)\ll r^{14}$. This leads to the estimate
\[  M_2(B)=\frac{3\tau_{\infty}\log (B^{1/6-2\eta})}{\zeta(4)}
  \sum_{q=1}^\infty q^{-8}\psi(q)\prod_{p\mid q}(1-p^{-4})^{-1}
+O_\eta(1).\]
Finally we note that
\begin{align*}
  \sum_{q=1}^{\infty}q^{-8}\psi(q)\prod_{p\mid q}(1-p^{-4})^{-1}
  &=\prod_p\left(1+\frac{p^{-16}\psi(p^2)+p^{-32}\psi(p^4)+\ldots}
  {1-p^{-4}}\right)\\
&=\prod_p\left(1+(1-p^{-1})\{p^{-2}+p^{-4}+\ldots\}\right)\\
  &=\prod_p\left(\frac{1-p^{-3}}{1-p^{-2}}\right)\\
  &=\frac{\zeta(2)}{\zeta(3)}.
\end{align*}
We therefore conclude that
\[M_2(B)=\frac{\zeta(2)}{\zeta(3)\zeta(4)}
3\tau_{\infty}\log (B^{1/6-2\eta})+O_\eta(1).\]
Lemma \ref{M2P} then follows.

\subsection{Conclusion}\label{s:final}

It is now time to bring Lemmas \ref{M1B} and \ref{M2P} together in
Lemma \ref{together}. This yields
\[N(\Omega;B)=  \frac{B\log B}{4\zeta(3)\zeta(4)}\tau_{\infty} +
O(\eta^{1/2}B\log B)+O_{\eta}(B).\]
This is an asymptotic formula
which holds for any $\eta\in(0,\tfrac{1}{100})$. 
Suppose that the error terms are
$E_1+E_2$, in which $|E_1|\le c_1\eta^{1/2}B\log B$, and
$|E_2|\le c_2(\eta)B$.
We claim that the error terms may be replaced by $o(B\log B)$. To show
this, we suppose that some small $\ve>0$ is given, and we proceed to
show that there is a $B(\ve)$ such that $|E_1+E_2|\le \ve B\log B$
whenever $B\ge B(\ve)$.  Let $\eta=\{\ve/(2c_1)\}^{2}$. Then 
$|E_1|\le\tfrac12\ve B\log B$ for every $B$. With this value of $\eta$
we then set
\[B(\ve)=\exp\{2c_2(\eta)/\ve\},\]
so that $|E_2|\le\tfrac12\ve B\log B$ for all $B\ge B(\ve)$. This
proves our claim. It therefore follows that 
\[N(\Omega;B)\sim cB\log B,\]
as $B\to \infty$, with 
\[c=\frac{\tau_{\infty}}{4\zeta(3)\zeta(4)}.\]

In order to complete the proof of Theorem \ref{t:main} it remains to
check that our leading constant agrees  with the prediction  
by Peyre \cite{Pey95}.  
According to Schindler \cite[\S 3]{schindler}, the Peyre constant is equal to 
\beql{pc}
\frac{1}{4\zeta(2)\zeta(3)}\cdot \tau_\infty \prod_p
\lim_{t\to\infty} p^{-7t}n(p^t),
\eeq
where $\tau_\infty$ is given  by \eqref{eq:tau_infty} and 
where $n(p^t)$ is the number of $(\x,\y)\in (\ZZ/p^t\ZZ)^8$ such that
$F(\x;\y)\equiv 0\bmod{p^t}$. If $t\ge 1$ we have
\begin{align*}
n(p^t)&=\sum_{j=0}^t \sum_{\substack{\y\bmod{p^t}\\ (\y,p^t)=p^j  }} 
\#\{ \x\in (\ZZ/p^t\ZZ)^4: F(\x;\y)\equiv 0\bmod{p^t}\}\\
&=\sum_{j=0}^{[t/2]} \sum_{\substack{\u\bmod{p^{t-j}}\\ (\u,p)=1  }} 
\hspace{-0.1cm}
\#\{ \x\in (\ZZ/p^t\ZZ)^4: F(\x;\u)\equiv 0\bmod{p^{t-2j}}\}+O(p^{6t}).
\end{align*}
Since $p\nmid \u$ the number of $\x\in (\ZZ/p^{t-2j})^4$ such that 
$p^{t-2j}\mid F(\x;\u)$ is $p^{3(t-2j)}$. 
Thus 
\begin{align*}
n(p^t)
&=\sum_{j=0}^{[t/2]} \sum_{\substack{\u\bmod{p^{t-j}}\\ (\u,p)=1  }} 
p^{3t+2j}+O(p^{6t})\\
&=\sum_{j=0}^{[t/2]} \{p^{4(t-j)}-p^{4(t-j-1)}\}p^{3t+2j}+O(p^{6t})\\
&=\{1-p^{-4}\}\sum_{j=0}^{[t/2]}p^{7t-2j}+O(p^{6t})\\
&=\{1+p^{-2}\}p^{7t}+O(p^{6t}).
\end{align*}
It follows that $p^{-7t}n(p^t)$ tends to $1+p^{-2}$, so that (\ref{pc}) is
\[\frac{1}{4\zeta(2)\zeta(3)}\cdot \tau_\infty \frac{\zeta(2)}{\zeta(4)}.\]
Thus our leading constant $c$ in Theorem \ref{t:main} agrees with the
Peyre constant.

\end{document}